\documentclass{article}

\usepackage[margin=1in,a4paper]{geometry}
\usepackage{amssymb,amsmath,,wasysym,amsthm} 
\usepackage{enumitem}
\usepackage{xcolor}
\usepackage{float}
\usepackage[colorlinks,linkcolor=blue!60!black,citecolor=blue!80!black,pagebackref,hypertexnames=false, breaklinks]{hyperref}
\usepackage{circuitikz}
\def\Xint#1{\mathchoice
    {\XXint\displaystyle\textstyle{#1}}%
    {\XXint\textstyle\scriptstyle{#1}}%
    {\XXint\scriptstyle\scriptscriptstyle{#1}}%
    {\XXint\scriptscriptstyle\scriptscriptstyle{#1}}%
    \!\int}
    \def\XXint#1#2#3{{\setbox0=\hbox{$#1{#2#3}{\int}$}
    \vcenter{\hbox{$#2#3$}}\kern-.5\wd0}}
    \def\fint{\Xint-}

\newcommand{\pip}{\varphi}

\newcommand{\eps}{{\varepsilon}}

\newcommand{\supp}{{\rm supp}\, }

\theoremstyle{plain}
\newtheorem{thm}{Theorem}[section]
\newtheorem{theorem}{Theorem}[section]
\newtheorem{lem}[thm]{Lemma}
\newtheorem{cor}[thm]{Corollary}
\newtheorem{prop}[thm]{Proposition}

\theoremstyle{definition}
\newtheorem{defn}[thm]{Definition}
\newtheorem{assump}[thm]{Assumption}

\theoremstyle{remark}
\newtheorem{remark}[thm]{Remark}
\newcommand{\bremark}{\begin{remark} \em}
\newcommand{\eremark}{\end{remark} }

\usepackage[textwidth=3.2cm]{todonotes}

\begin{document}


\title{Korevaar-Schoen $p$-energies and their $\Gamma$-limits on Cheeger spaces}

\author{Patricia Alonso Ruiz\footnote{Partly supported by the NSF grant DMS~2140664}, Fabrice Baudoin\footnote{Partly supported by the NSF grant  DMS~2247117.}}

\maketitle

\begin{abstract}
The paper studies properties of $\Gamma$-limits of Korevaar-Schoen $p$-energies on a Cheeger space. When $p>1$, this kind of limit provides a natural $p$-energy form that can be used to define a $p$-Laplacian, and whose domain is the Newtonian Sobolev space $N^{1,p}$. When $p=1$, the limit can be interpreted as a total variation functional whose domain is the space of BV functions. When the underlying space is compact, the $\Gamma$-convergence of the $p$-energies is improved to Mosco convergence for every $p \ge 1$.
\end{abstract}

\bigskip


\medskip


\tableofcontents
\newpage

\section{Introduction}

In their seminal paper~\cite{KS93}, Korevaar and Schoen developed a general theory of Sobolev spaces and harmonic maps between Riemannian manifolds to treat variational problems in that setting. Those Sobolev spaces $W^{1,p}$ were constructed in such a way that the \emph{p-energy} form
\begin{equation}\label{E:p-energy_KS}
E_p(f):=\sup_{\substack{g\in C_c(X)\\ 0\leq g \leq 1}}\limsup_{\varepsilon\to 0}\frac{1}{\varepsilon^p}\int_X \int_{B(x,\varepsilon)}g(x)\frac{d_g(f(x),f(y))^p}{\varepsilon^{n-1}}\,d{\rm vol}_g(x)
\end{equation}
would act as the natural seminorm of $W^{1,p}$. Above, $d_g$ and ${\rm vol}_g$ denote the Riemannian distance and volume, while $n$ is the dimension of the manifold. 
In principle, the expression~\eqref{E:p-energy_KS} makes sense in an arbitrary metric measure space $(X,d,\mu)$ and it is thus natural to study it beyond the Riemannian setting. In recent years, the body of literature exploring $(1,p)$-Sobolev spaces and $p$-energy forms in the context of abstract metric measure spaces has started to grow significantly; we refer to~\cite{ambrosiosurvey,ARB23,SW04,HKST15,Kig23,Bau22,Sha00} and references therein for an overview of available results. 

\medskip

%

How do $p$-energies arise in a natural way in a general metric measure space ? Along the lines of previous work by the authors in the case $p=2$, c.f.~\cite{ARB23}, one of the main results of the paper, Theorem~\ref{T:KS_as_Gamma}, provides the existence of a $p$-energy form $\mathcal{E}_p$ in $L^p(X,\mu)$ as the $\Gamma$-limit of a sequence of approximating $p$-energy functionals
\begin{equation}\label{E:def_Eprn_intro}
E_{p,r_n}(f):= \frac{1}{r_n^{p}}\int_X\fint_{B(x,r_n)} |f(y)-f(x)|^p d\mu(y) d\mu(x), \qquad f\in L^p(X,\mu),
\end{equation}
where $r_n\to 0$. Introduced by de Giorgi and Franzoni in~\cite{dGF75}, $\Gamma$-convergence was designed to study variational problems and it guarantees that a minimizer of the limiting functional $\mathcal{E}_p$ is the limit of a sequence of minimizers of each $E_{p,r_n}$. In addition, Theorem~\ref{T:KS_as_Gamma} shows that in the framework of a Cheeger space, the domain of the $p$-energy $\mathcal{E}_p$ is the Korevaar-Schoen type $(1,p)$-Sobolev space $KS^{1,p}(X)$. 
A Cheeger space is a doubling metric measure space that satisfies a $(p,p)$-Poincar\'e inequality for Lipschitz functions, c.f. Section~\ref{SS:Cheeger_spaces}. A forthcoming paper will handle Dirichlet spaces with sub-Gaussian heat kernel estimates that include among others Sierpinski carpets and gaskets, for which $p$-energies have been constructed via finite graph approximations in~\cite{HPS04,Kig23,MS23}.

\medskip

In the context of non-linear functionals such as Korevaar-Schoen $p$-energies, $\Gamma$-convergence is the typical convergence mode to consider. Motivated by the results in~\cite{ARB23}, the present paper also investigates the concept of Mosco convergence for sequences of functionals like $\{E_{p,r_n}\}_{n\geq 0}$ when the underlying space is compact, see Section~\ref{Mosco section}. The latter convergence is an extension of the original Mosco convergence for Dirichlet forms ($2$-energies), see~\cite{Mos94}. 


\medskip

Another aspect of $2$-energies (Dirichlet forms) that translates to $p$-energies is their correspondence with a family of Radon measures $\{\Gamma_p(f)\colon\,f\in KS^{1,p}(X)\}$ on the underlying space $X$. These measures have the property that the total $p$-energy $\mathcal{E}_p(f)$ coincides with the total mass of the space $\Gamma_p(f)(X)$. The technique used to construct these measures relies on a \emph{localization method} from the theory of $\Gamma$-convergence. The main idea is to consider first the functionals 
\[
E_{p,r_n}(f,U):= \frac{1}{r_n^{p}}\int_U\fint_{B(x,r_n)} |f(y)-f(x)|^p d\mu(y) d\mu(x), \qquad f\in L^p(X,\mu),\; U\subset X\text{ open,}
\]
that are localized versions of~\eqref{E:def_Eprn_intro}, and second to prove that their $\overline{\Gamma}$-limit is indeed a Radon measure $\Gamma_p(f)$,  which is the $p$-energy measure associated with $f$ and satisfies $\Gamma_p(f)=\mathcal{E}_p(f)$. A precise definition of $\overline{\Gamma}$-convergence is provided in Section~\ref{SS:Gamma_convergence}. The authors believe that the approach presents a novel application of the localization method in the context of $p$-energy measures on Cheeger spaces.
Desirable properties of the $p$-energy measure that carry over from the case $p=2$ include its absolute continuity with respect to the underlying measure and the possibility to obtain a $(p,p)$-Poincar\'e inequality with respect to it. 

\medskip

The paper is organized as follows: Section~\ref{S:MMSintro} briefly describes the main assumptions and concepts from $\Gamma$-convergence along with observations that will be applied repeatedly. The existence of $p$-energy forms as $\Gamma$-limits of the Korevaar-Schoen energies~\eqref{E:def_Eprn_intro} is proved in Section~\ref{S:KS_pforms} as well as some of their properties and a discussion of the associated $p$-Laplacian. The construction of the corresponding $p$-energies appears in Section~\ref{S:KS_pmeasures}, where also a $(p,p)$-Poincar\'e inequality is obtained and the absolute continuity of the measures for $p>1$ is proved. In the case $p=1$, the energy measures arising as $\bar{\Gamma}$-limits are shown to be uniformly comparable to the BV measures constructed by Miranda in~\cite{miranda}. Finally, Section~\ref{Mosco section} focuses on the case when the underlying space is compact. Here a Rellich-Kondrachov theorem and the stronger Mosco convergence of the Korevaar-Schoen energies~\eqref{E:def_Eprn_intro} are presented.

\section{Definitions and setup}\label{S:MMSintro}
The assumptions on the underlying space that we make throughout the paper correspond to what is often referred to as a \emph{Cheeger space} after the influential work of Cheeger in~\cite{Che99}. In this setup we investigate energy functionals that arise as $\Gamma$- and $\overline{\Gamma}$-limits of suitable sequences. These convergence types are reviewed in section~\ref{SS:Gamma_convergence}.

\subsection{Cheeger spaces}\label{SS:Cheeger_spaces}
Let $(X,d,\mu)$ denote a locally compact, complete, metric measure space, where $\mu$ is a Radon measure. Any open metric ball centered at $x\in X$ with radius $r>0$ is denoted by
\[
B(x,r):= \{ y \in X, d(x,y)<r \}.
\] 
When convenient, for a ball $B:=B(x,r)$ and $\lambda>0$, the ball $B(x,\lambda r)$ will be denoted by $\lambda B$.

\bigskip

Note that, in this setup, closed balls are compact and the space is separable. Thus, any open set can be expressed as a countable union of balls with rational radii which makes the space $\sigma$-finite. 
That property will play a role in proving Theorem~\ref{T:Gamma_measure}. 

\begin{assump}
The measure $\mu$ is doubling and positive in the sense that there exists a constant $C>0$ such that for every $x \in X, r>0$,
\begin{equation}\label{A:VD}
0< \mu (B(x,2r)) \le C \mu(B(x,r)) <\infty.\tag{$\rm VD$}
\end{equation}
\end{assump}

From the doubling property of $\mu$ it follows that there exist constants $C > 0$ and $0<Q\ <\infty$ such that
\begin{equation}\label{eq:mass-bounds}
 \frac{\mu(B(x,R))}{\mu(B(x,r))}
 \le C\left(\frac{R}{r}\right)^Q
\end{equation}
for any $0<r\le R$ and $x\in X$, see e.g.~\cite[Lemma 8.1.13]{HKST15}. 

\medskip

\begin{remark}\label{R:cover_partition}
Another useful consequence of the doubling property is the availability of maximally separated $\eps$-coverings with the bounded overlap property and subordinated Lipschitz partitions of unity, see~\cite[pp. 102-104]{HKST15}.  This means that for every $\lambda \ge 1$, there exists a constant $C>0$ such that for every $\varepsilon >0$, one can find a covering of $X$ by a family of balls $\{B_i^\eps:=B(x_i,\eps)\}_{i\geq 1}$ so that the family $\{B_i^{\lambda \eps}\}_{i\geq 1}$ satisfies
\[
\sum_{i\geq 1}\mathbf{1}_{B(x_i,\lambda \eps)}(x)<C
\]
for all $x\in X$.  A subordinated Lipschitz partition of unity is a family of $(C/\eps)$-Lipschitz functions $\pip_i^\eps$ $0\le \pip_i^\eps\le 1$ on $X$, $\sum_i\pip_i^\eps=1$ on $X$, and $\pip_i^\eps=0$ in $X\setminus B_i^{2 \eps}$. 
The importance of these tools will become more apparent in Section~\ref{SS:Observations}.
\end{remark}

\bigskip

The second main assumption is a $(p,p)$-Poincar\'e inequality with respect to Lipschitz functions. Recall that the Lipschitz constant of a function $f\in{\rm Lip}(X)$ is defined as
\begin{equation*}
(\mathrm{Lip} f )(y):=\limsup_{r \to 0^+} \sup_{x \in X, d(x,y) \le r} \frac{|f(x)-f(y)|}{r}.
\end{equation*}  

Throughout the paper we will consider an exponent $p \ge 1$ and make the following assumption.

\begin{assump}[$(p,p)$-Poincar\'e inequality with Lipschitz constants]\label{A:pPI_Lip}
 For any $f\in {\rm Lip}(X)$ and any ball $B(x,R)$ of radius $R>0$,
\begin{align}\label{E:pPI_Lip}
\int_{B(x,R)} | f(y) -f_{B(x,R)}|^p d\mu (y) \le C R^p \int_{B(x,\lambda R)} (\mathrm{Lip} f )(y)^p d\mu (y),
\end{align}
where
\begin{equation}\label{E:def_average}
f_{B(x,R)}:=\fint_{B(x,R)}f(y)\,d\mu(y):=\frac{1}{\mu(B(x,R))}\int_{B(x,R)}f(y)\,d\mu(y).
\end{equation}
The constants $C>0$ and $\lambda \ge 1$ in~\eqref{E:pPI_Lip} are independent from $x$, $R$ and $f$.
\end{assump}

\begin{remark}\label{R:pPI_Lip_vs_ug}
In the present setting, the $(p,p)$-Poincar\'e inequality~\eqref{E:pPI_Lip} is known to be equivalent to the $p$-Poincar\'e inequality with upper gradients, c.f.~\cite[Theorem 8.4.2]{HKST15}.
\end{remark}

\bigskip

A metric measure space $(X,d,\mu)$ as described above is often called a \emph{Cheeger space} or a \emph{PI space}. Here we are interested in studying the Korevaar-Schoen type energy functionals defined as
\begin{equation}\label{E:def_Epr}
E_{p,r}(f):= \frac{1}{r^{p}}\int_X\fint_{B(x,r)} |f(y)-f(x)|^p d\mu(y) d\mu(x)
\end{equation}
for any $f\in L^p(X,\mu)$, and the associated Korevaar-Schoen space 
\begin{equation}\label{E:def_KS1p}
KS^{1,p}(X):=\Big\{ f \in L^p(X,\mu), \, \limsup_{r \to 0^+} E_{p,r}(f) <+\infty \Big\}.  
\end{equation}

For any $f\in KS^{1,p}(X)$, we will also consider the localized energy counterpart
\begin{equation}\label{E:def_local_Epr}
E_{p,r}(f,U):= \frac{1}{r^{p}}\int_U\fint_{B(x,r)} |f(y)-f(x)|^p d\mu(y) d\mu(x)
\end{equation}
for any open $U\subset X$.

\begin{remark}\label{R:KS_equiv_weakUG}
For $p>1$, when equipped with the norm $\|f\|_{L^p(X,\mu)} +\sup_{r >0} E_{p,r}(f)^{1/p}$, the space $KS^{1,p}(X)$ coincides with the Newtonian Sobolev space $N^{1,p}(X)$ from~\cite{Sha00} with equivalent norm $\|f\|_{L^p(X,\mu)} +\| g_f \|_{L^p(X,\mu)}$, where $g_f$ is the minimal $p$-weak upper gradient of $f$. On the other hand for $p=1$,  when equipped with the norm $\|f\|_{L^1(X,\mu)} +\sup_{r >0} E_{1,r}(f)^{1/p}$, the space $KS^{1,1}(X)$ coincides with the space of bounded variation functions $BV(X)$ see for instance~\cite{Bau22,MMS16}. We refer to \cite{BV2} and \cite{miranda} for further descriptions of the space $BV(X)$ in that setting.
\end{remark}

\begin{remark}\label{R:Epr_continuity}
By virtue of~\cite[Lemma 3.1]{Bau22} the functional $E_{p,r}(f)\colon L^p(X,\mu)\to\mathbb{R}$ is continuous in $L^p$ for any fixed $r>0$ since
\[
E_{p,r}(f)\leq \frac{C}{r^p}\|f\|_{L^p(X,\mu)}.
\]
Analogous arguments yield the same property for $E_{p,r}(f,U)$ in~\eqref{E:def_local_Epr} for any fixed open set $U\subset X$.
\end{remark}


\subsection{$\Gamma$ and $\overline{\Gamma}$-convergence}\label{SS:Gamma_convergence}
The construction of Korevaar-Schoen $p$-energies and associated $p$-energy measures proposed in this paper rely on the concepts of $\Gamma$- and $\overline{\Gamma}$-convergence of functionals as presented in the monograph by Dal Maso~\cite[Chapter 9, Chapter 16]{DMas93}. While these apply in more generality, we review here the basic ideas in the context of functionals in $L^p(X,\mu)$ with $1 \le p<\infty$.

\begin{defn}\label{D:Gamma_limit}
A sequence of functionals $\{E_n\colon L^p(X,\mu)\to [-\infty,\infty]\}_{n\geq 1}$ is said to $\Gamma$-converge to a functional $E\colon L^p(X,\mu)\to [-\infty,\infty]$ if and only if
\begin{enumerate}[label=(\roman*)]
\item For every $f \in L^p(X,\mu)$ and every sequence $f_n$ that converges to $f$ strongly in $L^p(X,\mu)$ it holds that
\[
E(f) \le \liminf_{n \to +\infty} E_n (f_n).
\]
\item For every $f \in L^p(X,\mu)$ there exists a sequence $f_n$ converging to $f$ strongly in $L^p(X,\mu)$ such that
\[
\limsup_{n \to +\infty} E_n (f_n)\le E(f).
\]
\end{enumerate}
\end{defn}

\begin{remark}[Theorem 8.5~\cite{DMas93}]\label{R:cptness_Gamma}
Since $L^p(X,\mu)$ satisfies the second countability axiom, any sequence of functionals $\{E_n\}_{n>0}$ has a $\Gamma\!$-convergent subsequence.
\end{remark}

We would like to use the framework of $\Gamma$-convergence to study such limits for the Korevaar-Schoen type functionals introduced in Section~\ref{S:KS_pforms}. Because the latter are integral functionals, we will make use of a technique known in the theory of $\Gamma$-convergence as the \emph{localization method}: Given a functional $E\colon L^p(X,\mu)\to[-\infty,\infty]$, one considers its \emph{localized} version $E\colon L^p(X,\mu)\times\mathcal{O}\to[-\infty,\infty]$, where $\mathcal{O}$ denotes the family of all open subsets of the underlying space $X$. The convergence of these local versions is called $\overline{\Gamma}$-convergence and for functionals in $L^p(X,\mu)$ it can be characterized as follows.

\begin{defn}\label{D:def_Gamma_bar}
A sequence of functionals $\{E_n\colon L^p(X,\mu)\times\mathcal{O}\to [-\infty,\infty]\}_{n\geq 1}$ is said to $\overline{\Gamma}$-converge to $E\colon L^p(X,\mu)\times\mathcal{O}\to [-\infty,\infty]$ if and only if
\begin{enumerate}[label=(\roman*)]
\item For every $f \in \mathcal{X}$, for every $U \in \mathcal{O}$, and every sequence $f_n$ that converges strongly to $f$ in $L^p(X,\mu)$ it holds that
\[
E(f,U) \le \liminf_{n \to +\infty} E_n (f_n,U).
\]
\item For every $f \in L^p(X,\mu)$ and for every $U,V \in \mathcal O$ with $U\Subset V$ there exists a sequence $f_n$ converging strongly to $f$ in $L^p(X,\mu)$ such that
\[
\limsup_{n \to +\infty} E_n (f_n,U)\le E(f,V).
\]
\end{enumerate}
\end{defn}

\begin{remark}
The notation $U\Subset V$ means that the closure $\overline{U}$ is compact and satisfies $\overline{U} \subset V$.
\end{remark}

\begin{remark}\label{R:basic_props_gamma_bar}
As a consequence of the previous characterization, the $\overline{\Gamma}$-limit of a sequence is increasing, inner regular and lower semicontinuous, cf.~\cite[Remark 16.3]{DMas93}. Also the analogue of Remark~\ref{R:cptness_Gamma} is true for $\overline{\Gamma}$-convergence, cf.~\cite[Theorem 16.9]{DMas93}. 
\end{remark}

\bigskip

To avoid confusion in the terminology, let us point out that a functional $E\colon L^p(X,\mu)\times\mathcal{O}\to[-\infty,\infty]$ is said to be \emph{local} if for any $U\in\mathcal{O}$,
\begin{equation}\label{E:def_local_functional}
    E(f,U)=E(g,U)
\end{equation}
for all $f,g\in L^p(X,\mu)$ with $f|_U=g|_U$ $\mu$-a.e.

\subsection{Useful first observations}\label{SS:Observations}
This paragraph collects several consequences of the main assumptions that are used repeatedly throughout the paper. For any fixed $\varepsilon>0$, let $\{B_i^\varepsilon\}_{i\geq 1}$ denote a finite open cover with the finite overlap property and parameter $\lambda=5$ as in Remark~\ref{R:cover_partition}. Further, for any $f\in L^p(X,\mu)$,
\begin{equation}\label{E:def_Lip_approx}
    f_\varepsilon:=\sum_{i\geq 1}f_{B_i^\varepsilon}\varphi_i^\varepsilon
\end{equation}
defines a Lipschitz approximation of $f$.

\begin{prop}\label{P:Lip_const_estimate}
For any $\varepsilon>0$ and any $f\in L^p(X,\mu)$, the function $f_\varepsilon$ in~\eqref{E:def_Lip_approx} is locally Lipschitz with
\begin{equation}\label{E:Lip_const_estimate}
{\rm Lip}f_{\varepsilon}(z)\leq C\bigg(\frac{1}{\varepsilon^p}\fint_{5B_j^{\varepsilon}}\fint_{B(x,2\varepsilon)}|f(y)-f(x)|^pd\mu(y)\,d\mu(x)\bigg)^{1/p},\quad z \in B_j^\varepsilon,
\end{equation}
and $f_{\varepsilon}$ converges to $f$ in $L^p(X,\mu)$ as $\varepsilon\to 0^+$.
\end{prop}
\begin{proof}
    We prove first the estimate~\eqref{E:Lip_const_estimate}. Let $\varepsilon>0$. For any $x,y\in B_i^\varepsilon$,
    \begin{align*}
    |f_{\varepsilon}(x)-f_{\varepsilon}(y)|&\leq
    \sum_{2B_i^{\varepsilon}\cap 2B_j^{\varepsilon}\neq \emptyset}|f_{B_i^\varepsilon}-f_{B_j^\varepsilon}|\,|\varphi_i^\varepsilon(x)-\varphi_i^\varepsilon(y)|\\
    &\leq \frac{c}{\varepsilon}d(x,y)\sum_{2B_i^{\varepsilon}\cap 2B_j^{\varepsilon}\neq \emptyset}\Big|\fint_{B_i^\varepsilon}f(x)d\mu(x)-\fint_{B_j^\varepsilon}f(y)\,d\mu(y)\Big|\\
    &\leq \frac{c}{\varepsilon}d(x,y)\sum_{2B_i^{\varepsilon}\cap 2B_j^{\varepsilon}\neq \emptyset}\fint_{B_i^\varepsilon}\fint_{B(x,2\varepsilon)}|f(x)-f(y)|d\mu(y)d\mu(x)\\
    &\leq \frac{c}{\varepsilon}d(x,y)\sum_{2B_i^{\varepsilon}\cap 2B_j^{\varepsilon}\neq \emptyset}\bigg(\fint_{B_i^\varepsilon}\fint_{B(x,2\varepsilon)}|f(x)-f(y)|^pd\mu(y)d\mu(x)\bigg)^{1/p}.
\end{align*}
The finite overlap property finally implies~\eqref{E:Lip_const_estimate}. Also due to the finite overlap property of $\{B_i^\varepsilon\}_{i\geq 1}$, for any $x\in B_j^\eps$ it holds that
\begin{align*}
    |f_{\eps}(x)-f(x)|&=\Big| \sum_{i:B_i^{2\eps}\cap B_j^{2\eps}\ne\emptyset}\varphi_i^\eps (x)(f_{B_i^\eps}-f(x))\Big|\\
    &\leq \sum_{i:B_i^{2\eps}\cap B_j^{2\eps}\ne\emptyset}\Big|\fint_{B_i^\eps}(f(y)-f(x))\,d\mu(y)\Big|\\
    &\leq \sum_{i:B_i^{2\eps}\cap B_j^{2\eps}\ne\emptyset}\fint_{B_i^\eps}|f(y)-f(x)|\,d\mu(y)\\
    &\leq C\fint_{B(x,6\eps)}|f(x)-f(y)|d\mu(y),
\end{align*}
whence
\begin{align}
\|f_{\eps}-f\|_{L^p(X,\mu)}^p & 
    = \int_X |f_\varepsilon(x)-f(x)|^pd\mu(x)\notag\\
  &\leq \sum_{j\geq 1}\int_{B_j^\varepsilon}|f_\varepsilon(x)-f(x)|^pd\mu(x)\notag\\
  &\le C  \int_X\left( \fint_{B(x,6\eps)}|f(x)-f(y)|d\mu(y) \right)^pd\mu(x).\label{E:Lip_const_estimate_01}
\end{align}
By virtue of~\cite[Theorem 3.5.6]{HKST15}, the maximal function 
\[
Mf(x)=\sup_{r>0}\fint_{B(x,r)}|f|\,d\mu
\]
is bounded in $L^p(X,\mu)$
and therefore the average integral in~\eqref{E:Lip_const_estimate_01} is bounded uniformly on $\varepsilon$. Dominated convergence and the Lebesgue differentiation theorem, see e.g.~\cite[(3.4.10)]{HKST15}, finally imply
\begin{equation}\label{E:abs_cont_Gamma_p_04}
\lim_{\eps \to 0^+} \int_X\left( \fint_{B(x,6\eps)}|f(x)-f(y)|d\mu(y)\right)^p\,d\mu(x)=0.
\end{equation}
\end{proof}

A local version of the $(p,p)$-Poincar\'e inequality~\eqref{E:pPI_Lip} is provided in the next observation. Here, for an open set $U\in \mathcal{O}$ and $\lambda>0$ we denote by $U_\lambda$ a $\lambda$-neighborhood of $U$, i.e.
\[
U_\lambda=\left\{x \in U, d(x,U) < \lambda \right\}.
\]

\begin{prop}\label{P:for_local_Kumagai_Sturm}
There exists $\Lambda>1$ such that for any $r>0$, $U\in\mathcal {O}$ and $f\in L^p(X,\mu)$,
    \begin{equation}\label{E:local_Kumagai_Sturm_01}
     \frac{1}{r^p}\int_U\fint_{B(x,r)}|f(x)-f(y)|^pd\mu(y)\,d\mu(x)\leq C\int_{U_{\Lambda r}}|{\rm Lip }f|^pd\mu.
    \end{equation}
\end{prop}
\begin{proof}
Consider an open cover $\{B_i^\varepsilon\}_{i\geq 1}$ which as before satisfies the bounded overlap property. Then,
    \begin{align}
         &\frac{1}{\varepsilon^p}\int_U\fint_{B(x,\varepsilon)}|f(x)-f(y)|^pd\mu(y)\,d\mu(x)\notag\\
         &\leq \frac{1}{\varepsilon^p}\sum_{B_i^\varepsilon\cap U\neq\emptyset}\int_{B_i^\varepsilon}\fint_{B(x,\varepsilon)}|f(x)-f(y)|^pd\mu(y)\,d\mu(x)\notag\\
         &\leq \frac{2^{p-1}}{\varepsilon^p}\sum_{B_i^\varepsilon\cap U\neq\emptyset}\int_{B_i^\varepsilon}\fint_{B(x,\varepsilon)}|f(x)-f_{B_i^\varepsilon}|^pd\mu(y)\,d\mu(x)\\
         &+\frac{2^{p-1}}{\varepsilon^p}\sum_{B_i^\varepsilon\cap U\neq\emptyset}\int_{B_i^\varepsilon}\fint_{B(x,\varepsilon)}|f(y)-f_{B_i^\varepsilon}|^pd\mu(y)\,d\mu(x).\label{E:local_limsup_vs_liminf_00}
     \end{align}
     For the first term in the latter expression, the $(p,p)$-Poincar\'e inequality~\eqref{E:pPI_Lip} yields
     \begin{equation}\label{E:local_limsup_vs_liminf_01}
         \int_{B_i^\varepsilon}\fint_{B(x,\varepsilon)}|f(x)-f_{B_i^\varepsilon}|^pd\mu(y)\,d\mu(x)
         \leq \int_{B_i^\varepsilon}|f(x)-f_{B_i^\varepsilon}|^pd\mu(x)\leq C\varepsilon^p\int_{\Lambda B_i^\varepsilon}|{\rm Lip }f|^pd\mu. 
     \end{equation}
     For the second term in~\eqref{E:local_limsup_vs_liminf_00}, Fubini and the volume doubling property imply
     \begin{align}
       &\int_{B_i^\varepsilon}\fint_{B(x,\varepsilon)}|f(y)-f_{B_i^\varepsilon}|^pd\mu(y)\,d\mu(x)\notag\\
       &\leq C\int_{2B_i^\varepsilon}\int_{B(y,\varepsilon)}\frac{1}{\mu(B(x,\varepsilon)}|f(y)-f_{B_i^\varepsilon}|^pd\mu(x)\,d\mu(y)\notag\\
       &\leq C\int_{2B_i^\varepsilon}|f(y)-f_{B_i^\varepsilon}|^p\int_{B(y,\varepsilon)}\frac{1}{\mu(B(y,\varepsilon)}d\mu(x)\,d\mu(y)=C\int_{2B_i^\varepsilon}|f(y)-f_{B_i^\varepsilon}|^pd\mu(y)\notag\\
       &\leq C\int_{2B_i^\varepsilon}|f(y)-f_{2B_i^\varepsilon}|^pd\mu(y)+C\int_{2B_i^\varepsilon}|f_{2B_i^\varepsilon}-f_{B_i^\varepsilon}|^pd\mu(y).\label{E:local_limsup_vs_liminf_02}
     \end{align}
     The first term in~\eqref{E:local_limsup_vs_liminf_02} is bounded using the $(p,p)$-Poincar\'e inequality as in~\eqref{E:local_limsup_vs_liminf_01}. For the second one applies H\"older and the $p$-Poincar\'e inequality to obtain
     \begin{align}
         \int_{2B_i^\varepsilon}|f_{2B_i^\varepsilon}-f_{B_i^\varepsilon}|^pd\mu(y)&\leq C\mu(2B_i^\varepsilon)\Big|\int_{B_i^\varepsilon}(f_{2B_i^\varepsilon}-f(y))\,d\mu(y)\Big|^p\notag\\
         &\leq C\Big|\int_{B_i^\varepsilon}(f_{2B_i^\varepsilon}-f(y))\,d\mu(y)\Big|^p\notag\\
         &\leq C\mu(B_i^\varepsilon)^{p/q}\int_{B_i^\varepsilon}|f(y)-f_{2B_i^\varepsilon}|^pd\mu(y)\notag\\
         &\leq C\int_{B_i^\varepsilon}|f(y)-f_{2B_i^\varepsilon}|^pd\mu(y)\notag\\
         &\leq C\varepsilon^p\int_{2\Lambda B_i^\varepsilon}|{\rm Lip }f|^pd\mu.\label{E:local_limsup_vs_liminf_03}
     \end{align}
     Plugging the estimates from~\eqref{E:local_limsup_vs_liminf_01},~\eqref{E:local_limsup_vs_liminf_02} and~\eqref{E:local_limsup_vs_liminf_03} into~\eqref{E:local_limsup_vs_liminf_00} it follows that
    \begin{equation*}
    \frac{1}{\varepsilon^p}\int_U\fint_{B(x,\varepsilon)}|f(x)-f(y)|^pd\mu(y)\,d\mu(x)\leq C\sum_{B_i^\varepsilon\cap U\neq\emptyset}\int_{2\Lambda B_i^\varepsilon}|{\rm Lip }f|^pd\mu\leq C \int_{U_{2\Lambda \varepsilon}}|{\rm Lip }f|^pd\mu,
    \end{equation*}
as claimed in~\eqref{E:local_Kumagai_Sturm_01}.
\end{proof}
\section{Korevaar-Schoen $p$-energy forms}\label{S:KS_pforms}
The starting point towards constructing a $p$-energy form associated with the space $(X,d,\mu)$ is the sequence of Korevaar-Schoen type energy functionals~\eqref{E:def_Epr}. 

\subsection{Existence of $\Gamma$-limits}
The following lemma is key to guarantee the existence of a $\overline{\Gamma}$-limit of the localized functionals $\{E_{p,r}(f,U)\}_{r>0}$ and ultimately of a $\Gamma$-limit of the sequence $\{E_{p,r}\}_{r>0}$ in Theorem~\ref{T:KS_as_Gamma}.

\begin{lem}\label{L:local_Kumagai-Sturm_cond}
    Let $\{\varepsilon_n\}_{n\geq 1}$ with $\varepsilon_n>0$ and $\varepsilon_n\to 0$. There exist $C>0$ and $\Lambda >1$ independent of $\{\varepsilon_n\}_{n\geq 1}$ such that 
    \begin{equation}\label{E:local_Kumagai-Sturm_cond1}
    E_{p,r}(f,U) \le C \liminf_{n\to\infty}E_{p,\varepsilon_n}(f_n,U_{ \Lambda r})
    \end{equation}
     for every $r>0$, any $U\in\mathcal{O}$, all $f\in L^p(X,\mu)$ and $\{f_n\}_{n\geq 1}\subset L^p(X,\mu)$ with $f_n\xrightarrow{L^p}f$. In particular,
    \begin{equation}\label{E:local_Kumagai-Sturm_cond}
        \limsup_{r\to 0^+}E_{p,r}(f,U)\leq C \limsup_{r\to0^+}\liminf_{n\to\infty}E_{p,\varepsilon_n}(f_n,U_{ r})
    \end{equation}
    for every $r>0$, any $U\in\mathcal{O}$, all $f\in L^p(X,\mu)$ and $\{f_n\}_{n\geq 1}\subset L^p(X,\mu)$ with $f_n\xrightarrow{L^p}f$.
\end{lem}
\begin{proof}
    Let $U\in\mathcal{O}$ and consider a sequence $\{f_n\}_{n\geq 1}$ that converges to some $f\in L^p(X,\mu)$ in $L^p$. Further, let $f_{n,\varepsilon}$ denote the Lipschitz approximation~\eqref{E:def_Lip_approx} of $f_n$. In view of Proposition~\ref{P:for_local_Kumagai_Sturm}, for any $r>0$ it holds that 
    \begin{equation}
    \frac{1}{r^p}\int_U\fint_{B(x,r)}|f_{n,\varepsilon}(x)-f_{n,\varepsilon}(y)|^pd\mu(y)\,d\mu(x)\leq C \int_{U_{\Lambda r}}|{\rm Lip}f_{n,\varepsilon}|^pd\mu.
\end{equation}
    Applying Proposition~\ref{P:Lip_const_estimate} further yields
    \begin{align}
        \int_{U_{\Lambda r}}|{\rm Lip}f_{n,\varepsilon}|^pd\mu&\leq \sum_{i,B_i^\varepsilon\cap U_{\Lambda r}\neq\emptyset}\int_{B_i^\varepsilon}|{\rm Lip}f_{n,\varepsilon}|^pd\mu\notag\\
        &\leq C\sum_{i,B_i^\varepsilon\cap U_{\Lambda r}\neq\emptyset} \int_{B_i^\varepsilon}\frac{1}{\varepsilon^p}\fint_{5B_i^\varepsilon}\fint_{B(z,2\varepsilon)}|f_{n}(z)-f_n(y)|^pd\mu(y)\,d\mu(z)\,d\mu(x)\notag\\
    &\leq C\sum_{i, B_i^\varepsilon \cap U_{\Lambda r} \neq \emptyset  } \int_{5B_i^\varepsilon}\frac{1}{\varepsilon^p}\fint_{B(z,2\varepsilon)}|f_n(z)-f_n(y)|^pd\mu(y)\,d\mu(z)\notag\\
    &\leq \frac{C}{\varepsilon^p}\int_{U_{\Lambda r +10\varepsilon}}\fint_{B(z,2\varepsilon)}|f_n(z)-f_n(y)|^pd\mu(y)\,d\mu(z).\label{E:local_Kumagai_Sturm_cond_01}
    \end{align}
Thus, for any $r>0$,
\begin{equation}\label{E:local_Kumagai_Sturm_cond_02}
    \frac{1}{r^p}\int_U\fint_{B(x,r)}|f_{n,\varepsilon}(x)-f_{n,\varepsilon}(y)|^pd\mu(y)\,d\mu(x)\leq \frac{C}{\varepsilon^p} \int_{U_{\Lambda r +10\varepsilon}}\fint_{B(z,2\varepsilon)}|f_n(z)-f_n(y)|^pd\mu(y)\,d\mu(z)
\end{equation}
and substituting $\varepsilon$ by $\varepsilon_n/2$ with $\varepsilon_n\to 0$ we obtain
\begin{align}
    &\liminf_{n\to\infty}\frac{1}{r^p}\int_U\fint_{B(x,r)}|f_{n,\varepsilon_n/2}(x)-f_{n,\varepsilon_n/2}(y)|^pd\mu(y)\,d\mu(x)\notag\\
    &\leq C \liminf_{n\to\infty}E_{p,\varepsilon_n}(f_n,U_{2\Lambda r+10\varepsilon_n/2})\notag
\end{align}
Let now $\eta>0$ and $N \ge 0$ such that for $n \ge N$, $\varepsilon_n  \le r$. We have then for $n \ge N$,
\begin{align}\label{E:local_Kumagai_Sturm_cond_03}
E_{p,\varepsilon_n}(f_n,U_{2\Lambda r+10\varepsilon_n/2}) \le E_{p,\varepsilon_n}(f_n,U_{2\Lambda r+10r/2})
\end{align}
On the other hand, due to the continuity of the functional $E_{p,r}(.,U)$ for fixed $U$, c.f. Remark~\ref{R:Epr_continuity}, it follows that 
\begin{equation}\label{E:local_Kumagai_Sturm_cond_04}
\liminf_{n\to\infty}E_{p,r}(f_{n,\varepsilon_n/2}, U)=E_{p,r}(f,U)    
\end{equation}
for any $U\in\mathcal{O}$. Together with~\eqref{E:local_Kumagai_Sturm_cond_03} that implies 
\begin{equation}\label{E:}
    E_{p,r}(f,U)\leq \liminf_{n\to\infty}E_{p,\varepsilon_n}(f_n,U_{(2\Lambda +10/2)r})
\end{equation}
as we wanted to prove.
\end{proof}

\begin{remark}
Lemma \ref{L:local_Kumagai-Sturm_cond} is stated for sequences, however an identical proof yields that there exist $C>0$ and $\Lambda >1$  such that 
    \begin{equation}\label{E:local_Kumagai-Sturm_cond3}
    E_{p,r}(f,U) \le C \liminf_{\varepsilon \to 0}E_{p,\varepsilon}(f_\varepsilon ,U_{ \Lambda r})
    \end{equation}
for every $r>0$, any $U\in\mathcal{O}$, all $f\in L^p(X,\mu)$ and $\{f_\varepsilon\}_{\varepsilon>0}\subset L^p(X,\mu)$ with $f_\varepsilon \xrightarrow{L^p}f$ when $\varepsilon \to 0$.
\end{remark}

\medskip

Applying \eqref{E:local_Kumagai-Sturm_cond1} with  $U=X$ yields the following refinement of the property $\mathcal{P}(p,1)$ in~\cite[Definition 4.5]{Bau22}.

\begin{lem}\label{L:Kumagai-Sturm_condition}
Let $\{r_n\}_{n\geq 0}$ with $r_n >0$ and $\lim_{n\to\infty}r_n=0$. There exists a constant $C>0$ independent of the sequence $\{r_n\}_{n\geq 0}$ such that
\[
  \sup_{r>0} E_{p,r}(f) \le  C \liminf_{n \to +\infty}  E_{p,r_n}(f_n)
\]
for all $f\in L^{p}(X,\mu)$ and all $\{f_n\}_{n\geq 1}\subset L^p(X,\mu)$ with $f_n\to f$ in $L^p(X,\mu)$.
\end{lem}

We finally arrive at the main result of this section, that provides the existence of a functional $\mathcal{E}_p$ which will be our natural candidate for an energy form.

\begin{theorem}\label{T:KS_as_Gamma}
There exists a positive sequence $\{r_n\}_{n\geq 1}$ converging to zero such that the $\Gamma$-limit
\begin{equation}\label{E:KS_as_Gamma}
    \mathcal{E}_p(f):=\Gamma\text{-}\!\!\lim_{n\to\infty}E_{p,r_n}(f)
\end{equation}
exists. Moreover,
\begin{equation*}
KS^{1,p}(X)=\{ f \in L^p(X,\mu)\colon \mathcal{E}_p(f)<\infty \}
\end{equation*}
and there exists $C>0$ such that for every $f \in KS^{1,p} (X)$
\begin{equation}\label{E:KS_comp_Gamma}
C  \sup_{r >0} E_p(f,r) \le \mathcal{E}_p(f) \le  \liminf_{r\to 0^+} E_{p,r}(f).
\end{equation}
\end{theorem}

\begin{proof}
Let $\{r_n\}_{n\geq 1}$ be the sequence from Lemma~\ref{L:Kumagai-Sturm_condition}. In view of Remark~\ref{R:cptness_Gamma}, the associated sequence $\{E_{p,r_n}(f)\}_{n\geq 1}$ has a $\Gamma$-convergent subsequence, which for simplicity we still denote same. Set $\mathcal{E}_p(f):=\Gamma\text{-}\lim\limits_{n\to\infty}E_{p,r_n}(f)$. Due to the characterization of $\Gamma$-convergence, for any sequence $\{f_n\}_{n\geq 1}$ that converges strongly to $f$ in $L^p(X,\mu)$ it holds that
\[
\mathcal{E}_p(f)\leq \liminf_{n\to\infty}E_{p,r_n}(f_n).
\]
Further, applying Lemma~\ref{L:Kumagai-Sturm_condition} with the sequence $\{f_n\}_{n\geq 1}$ from the characterization of $\Gamma$-convergence we obtain
\[
\sup_{r>0}E_{p,r}(f)\leq C\liminf_{n\to\infty}E_{p,r_n}(f_n)\leq C\limsup_{n\to\infty}E_{p,r_n}(f_n)\leq \mathcal{E}_p(f).
\]
\end{proof}

A consequence of the latter theorem is the density of Lipschitz functions, which can be regarded as a regularity result on the form $\mathcal E_p$ . 

\begin{cor}\label{C:Lip_density}
The set ${\rm Lip}_{\rm loc}(X)\cap C_c(X)$ is dense in $L^p(X,\mu)$ with respect to $(\mathcal{E}_p(\cdot,\cdot)+\|\cdot\|_{L^p(X,\mu)})^{1/p}$.
\end{cor}

\begin{proof}
In view of Remark~\ref{R:KS_equiv_weakUG} and~\eqref{E:KS_comp_Gamma}, the assertion follows from~\cite[Theorem 8.2.1]{HKST15}. 
\end{proof}

\begin{remark}
The question about uniqueness of the $\Gamma$-limit point for the functionals $E_{p,r}(f)$ in any Cheeger space is still open. 
Yet, if the space $(X,d,\mu)$ is ${\rm RCD}(0,N)$, it follows from~\cite{Gor22}, see also \cite{HP21}, that the $\Gamma$-limit from Theorem \ref{T:KS_as_Gamma} is independent of the subsequence $r_n$. Moreover, for $p>1$, 
for every $f \in KS^{1,p}(X)$
\[
\mathcal{E}_p (f)=\lim_{r \to 0} E_{p,r}(f)= C_{N,p} \int_X g_f^p d\mu=C_{N,p} \int_X \| Df\|^2 d\mu,
\]
where $C_{N,p}$ is a universal constant, $g_f$ is the minimal $p$-weak upper gradient of $f$  and $ \| Df\|$ is defined from a Cheeger differential structure,  see~\cite[Section 2.3]{MMS16}. In the case $p=1$ still in ${\rm RCD}(0,N)$ setting, for every $f \in KS^{1,p}(X)$
\[
\mathcal{E}_p (f)=\lim_{r \to 0} E_{p,r}(f)= C_{N,1} \| Df\|(X),
\]
where $\| Df\|(X)$ is the total variation of $f$ as defined in Section \ref{SS:BV measure}.
\end{remark}

\subsection{Some properties of the $p$-energy form}

The functional $\mathcal{E}_p(f)$ may be thought of as an analogue of the Euclidean $p$-energy form $\int_{\mathbb R^n} | \nabla f |^p dx$. Our next goal will be to define the analogue of  $\int_{\mathbb R^n} | \nabla f |^{p-2} \left \langle \nabla f,\nabla g \right \rangle dx$, which will in particular be used to define a $p$-Laplacian when $p>1$. 

\medskip

The next paragraphs present several key properties that allow to extend the functional $(\mathcal{E}_p,KS^{1,p}(X))$, in~\eqref{E:KS_as_Gamma} to a form $\mathcal{E}_p\colon KS^{1,p}(X)\times KS^{1,p}(X)\to\mathbb{R}$. 
%

\begin{lem}
$(\mathcal{E}_p,KS^{1,p}(X))$ is a convex functional.
\end{lem}

\begin{proof}
Let $f,g\in KS^{1,p}(X)$ and $\lambda \in [0,1]$. 
Further, let $\{f_n\}_{n\geq 1}$ and $\{g_n\}_{n\geq 1}$ be sequences as those from Definition~\ref{D:Gamma_limit}. 
In particular, the sequence $\lambda f_n +(1-\lambda) g_n$ converges to $\lambda f+(1-\lambda)g$ in $L^p(X,\mu)$. 
Therefore, by definition of $\Gamma$-convergence, 
\[
\mathcal{E}_p (\lambda f +(1-\lambda)g)\le \liminf_{n  \to +\infty}   E_{p}(\lambda f_n +(1-\lambda)g_n,r_n).
\]
Moreover, note that the functionals $E_p(\lambda f_n +(1-\lambda)g_n,r_n)$ are convex, whence
\[
\mathcal{E}_p (\lambda f +(1-\lambda)g)\le \liminf_{n  \to +\infty} \left(  \lambda  E_{p}( f_n ,r_n)+(1-\lambda)E_{p}(g_n ,r_n) \right)
\]
and thus
\[
\mathcal{E}_p (\lambda f +(1-\lambda)g) \le \lambda \mathcal{E}_p (f) +(1-\lambda)\mathcal{E}_p(g).
\]
\end{proof}

The next property is called ``Leibniz rule'' in~\cite[Theorem 6.25]{MS23}.

\begin{prop}\label{P:pseudo_Leibniz}
For any $f,g\in KS^{1,p}(X) \cap L^\infty(X,\mu)$, then $fg \in KS^{1,p}(X)$ and
\begin{equation}\label{E:pseudo_Leibniz}
    \mathcal{E}_p(fg)\leq 2^{p-1}\big(\mathcal{E}_p(f)\|g\|_{L^\infty(X,\mu)}+\mathcal{E}_p(g)\|f\|_{L^\infty(X,\mu)}\big).
\end{equation}
\end{prop}
\begin{proof}
    Let $f,g\in KS^{1,p}(X) \cap L^\infty(X,\mu)$ and consider the sequences $\{f_n\}_{n\geq 1}$, and $\{g_n\}_{n\geq 1}$ in $L^p(X,\mu)$ that satisfy $f_n\xrightarrow{L^p}f$, $g_n\xrightarrow{L^p}g$ and 
    \begin{equation}\label{E:pseudo_Leib_1}
    \limsup_{n\to\infty}E_{p,r_n}(f_n)\leq \mathcal{E}_p(f)\qquad\limsup_{n\to\infty}E_{p,r_n}(g_n)\leq \mathcal{E}_p(g),
    \end{equation}
    c.f. Definition~\ref{D:Gamma_limit}.  Let now $\varepsilon >0$ and consider the sequence
    \begin{align}
    \tilde{f}_n(x)=
    \begin{cases}
    f_n(x), \quad \text{if } |f_n(x)| \le \| f \|_{L^\infty(X,\mu)} +\varepsilon \\
    \| f \|_{L^\infty(X,\mu)} +\varepsilon, \quad \text{if } f_n(x) > \| f \|_{L^\infty(X,\mu)} +\varepsilon \\
    -\| f \|_{L^\infty(X,\mu)} -\varepsilon, \quad  \text{if } f_n(x) <- \| f \|_{L^\infty(X,\mu)} -\varepsilon.
    \end{cases}
      \end{align}
      A sequence $\tilde{g}_n$ is defined in terms of $g_n$ analogously. 
    We note that $\tilde f_n\xrightarrow{L^p}f$, $\tilde g_n\xrightarrow{L^p}g$ and moreover that $ \| \tilde f_n \|_{L^\infty(X,\mu)} \le \| f \|_{L^\infty(X,\mu)} +\varepsilon$, $ \| \tilde g_n \|_{L^\infty(X,\mu)} \le \| g \|_{L^\infty(X,\mu)} +\varepsilon$.
    
    Since $\tilde{f}_n \tilde{g}_n\xrightarrow{L^p}fg$, the definition of $\Gamma$-convergence  yields
    \begin{equation}\label{E:pseudo_Leib_2}
        \mathcal{E}_p(fg)\leq \liminf_{n\to\infty} E_{p,r_n}(\tilde f_n \tilde g_n).
    \end{equation}
    Moreover,
    \begin{align*}
        E_{p,r_n}(\tilde f_n \tilde g_n)&=\frac{1}{r_n^p}\int_X\fint_{B(x,r_n)}|\tilde f_n \tilde g_n(x)-\tilde f_n \tilde g_n(y)|^pd\mu(y)\,d\mu(x)\\
        &=\frac{1}{r_n^p}\int_X\fint_{B(x,r_n)}|g_n(y)\big(\tilde f_n(x)-\tilde f_n(y)\big)+\tilde f_n(x)\big(\tilde g_n(x)-\tilde g_n(y)\big)|^pd\mu(y)\,d\mu(x)\\
        &\leq \frac{2^p}{r_n^p}\int_X\fint_{B(x,r_n)}|\tilde g_n(y)|^p|\tilde f_n(x)-\tilde f_n(y)|^pd\mu(y)\,d\mu(x)\\
        &+\frac{2^p}{r_n^p}\int_X\fint_{B(x,r_n)}|\tilde f_n(x)|^p|\tilde g_n(x)-\tilde g_n(y)|^pd\mu(y)\,d\mu(x)\\
         &\leq \frac{2^p}{r_n^p}\int_X\fint_{B(x,r_n)}|\tilde g_n(y)|^p| f_n(x)- f_n(y)|^pd\mu(y)\,d\mu(x)\\
        &+\frac{2^p}{r_n^p}\int_X\fint_{B(x,r_n)}|\tilde f_n(x)|^p| g_n(x)- g_n(y)|^pd\mu(y)\,d\mu(x)\\
        &\leq 2^p\big(\|\tilde g_n\|_{L^\infty(X,\mu)}E_{p,r_n}(f_n)+\|\tilde f_n\|_{L^\infty(X,\mu)}E_{p,r_n}(g_n)\big)\\
         &\leq 2^p\big((\|\tilde g\|_{L^\infty(X,\mu)}+\varepsilon)E_{p,r_n}(f_n)+(\|\tilde f\|_{L^\infty(X,\mu)}+\varepsilon)E_{p,r_n}(g_n)\big).
    \end{align*}
    Plugging the latter into~\eqref{E:pseudo_Leib_2} and using~\eqref{E:pseudo_Leib_1} we finally get~\eqref{E:pseudo_Leibniz} since $\varepsilon$ is arbitrary.
\end{proof}

\begin{prop}\label{P:comp_Lip}
    For any $f\in KS^{1,p}(X)$ and any $1$-Lipschitz function $\varphi\in C(\mathbb{R})$, $\varphi{\circ}f\in KS^{1,p}(X)$ and
    \begin{equation}\label{E:comp_Lip}
        \mathcal{E}_p(\varphi{\circ}f)\leq \mathcal{E}_p(f).
    \end{equation}
\end{prop}
\begin{proof}
    Take the same sequence as in~\eqref{E:pseudo_Leib_1} and note that $\varphi{\circ}f_n\xrightarrow{L^p}\varphi{\circ}f$. 
    Further,
    \begin{align*}
        E_{p,r_n}(\varphi{\circ}f_n)&=\frac{1}{r_n^p}\int_X\fint_{B(x,r_n)}|\varphi{\circ}f_n(x)-\varphi{\circ}f_n(y)|^pd\mu(y)\,d\mu(x)\\
        &\leq \frac{1}{r_n^p}\int_X\fint_{B(x,r_n)}|f_n(x)-f_n(y)|^pd\mu(y)\,d\mu(x)=E_{p,r_n}(f_n).
    \end{align*}
    By definition of $\Gamma$-convergence and~\eqref{E:pseudo_Leib_1}, the latter implies
    \begin{equation*}
        \mathcal{E}_p(\varphi{\circ}f)\leq \liminf_{n\to\infty} E_{p,r_n}(\varphi{\circ}f_n)\leq \limsup_{n\to\infty} E_{p,r_n}(\varphi{\circ}f_n)\leq \limsup_{n\to\infty} E_{p,r_n}(f_n)\leq \mathcal{E}_p(f)
    \end{equation*}
    as we wanted to prove.
\end{proof}
The next two lemmas provide rigorous arguments leading to the understanding of $\mathcal{E}_p(f,g)$ as the derivative of $t\mapsto\frac{1}{p}\mathcal{E}_p(f+tg)$ at $t=0$. This approach, suggested in~\cite{SW04}, will allow us to pass from the functional $\mathcal{E}_p(f)$ to a $p$-energy \emph{form} $\mathcal{E}_p(f,g)$. 

\bigskip

The first lemma guarantees that the left and right limit coincide.
\begin{lem}\label{L:same_lim_RL}
Assume $p>1$. For any $f,g\in KS^{1,p}(X)$,
\begin{equation}\label{E:same_lim_RL}
    \lim_{t \to 0} \frac{\mathcal{E}_p (f+tg) +\mathcal{E}_p (f-tg)- 2 \mathcal{E}_p (f)}{t}=0.
\end{equation}
\end{lem}

\begin{proof}
Assume first $p \ge 2$. Taylor's expansion provides for $x,y \in \mathbb{R}$ and $t \in [-1,1]$,
\begin{equation}\label{E:Taylor_abs_p}
|x+ty|^p=|x|^p+p y \mathrm{sign}(x) |x|^{p-1} t+t^2R_1(t,x,y),
\end{equation}
where the remainder $R_1(t,x,y)$ satisfies $|R_1(t,x,y)| \le C |y|^2(|x|^{p-2}+|y|^{p-2})$. Thus we have
\begin{align}\label{eq-par}
|x+ty|^p+|x-ty|^p=2|x|^p+t^2R_2(t,x,y)
\end{align}
where $R_2(t,x,y)$ satisfies $|R_2(t,x,y)| \le C |y|^2(|x|^{p-2}+|y|^{p-2})$.

\medskip

Let now $f,g \in KS^{1,p}(X)$. Let $f_n,g_n \in L^p(X,\mu)$ be such that $f_n\to f$ in $L^p$, $g_n\to g$ in $L^p$, and $E_{p,r_n}(f_n)\to \mathcal{E}_{p}(f)$, $E_p(g_n,r_n)\to \mathcal{E}_p (g) $. Using~\eqref{eq-par} one obtains
\begin{align*}
E_{p,r_n}(f_n+tg_n)+E_{p,r_n}(f_n-tg_n)=2E_{p,r_n}(f_n)+t^2R_n(t),
\end{align*}
where H\"older's inequality shows that the remainder term $R_n(t)$ satisfies 
\[
|R_n(t)| \le C(E_{p,r_n}(f_n)^{1-2/p}E_{p,r_n}(g_n)^{2/p}+E_{p,r_n}(v_n)). 
\]
By $\Gamma$-convergence one has then
\begin{align*}
\mathcal{E}_p (f+tg) +\mathcal{E}_p (f-tg) &\le  \liminf_{n \to +\infty} E_{p,r_n}(g_n+tg_n)+\liminf_{n \to +\infty} E_{p,r_n}(f_n-tg_n) \\
 &\le  \liminf_{n \to +\infty} \left[ E_{p,r_n}(f_n+tg_n)+ E_{p,r_n}(f_n-tg_n) \right] \\
 &\le \liminf_{n \to +\infty} \left[ 2E_{p,r_n}(f_n)+t^2R_n(t) \right] \\
 &\le 2 \mathcal{E}_p (f)+t^2 \sup_n |R_n(t)|
\end{align*}
Therefore,
\[
\limsup_{t \to 0} \frac{\mathcal{E}_p (f+tg) +\mathcal{E}_p (f-tg)- 2 \mathcal{E}_p (f)}{t} \le 0.
\]
On the other hand, the convexity of the functional $\mathcal{E}_p$ implies that
\[
\mathcal{E}_p (f+tg) +\mathcal{E}_p (f-tg)- 2 \mathcal{E}_p (f) \ge 0
\]
from which we conclude
\[
\lim_{t \to 0} \frac{\mathcal{E}_p (f+tg) +\mathcal{E}_p (f-tg)- 2 \mathcal{E}_p (f)}{t} = 0.
\]
In the case $1<p<2$, for $x,y \in \mathbb{R}$ and $t \in [-1,1]$ we have
\[
|x+ty|^p=|x|^p+p y\,{\rm sign}(x) |x|^{p-1} t+t^{2}R_1(t,x,y),
\]
where the Taylor remainder $R_1(t,x,y)$ satisfies $|R_1(t,x,y)| \le C |y|^p$. The proof proceeds then exactly as before.
\end{proof}

The second lemma in addition provides an expression of the $p$-energy form in terms of suitable convergent sequences of functions.

\begin{lem}\label{L:def_Ep_fg}
Assume $p>1$. For any $f,g \in KS^{1,p}(X)$ the limit  
\[
\mathcal{E}_p (f,g):=\frac{1}{p} \lim_{t \to 0^+} \frac{\mathcal{E}_p (f+tg) - \mathcal{E}_p (f)}{t}
\]
exists. Moreover, for any $f,g\in KS^{1,p}(X)$ there exists a sequence $\{f_n\}_{n\geq 1}\subset L^p(X,\mu)$ with $f_n\xrightarrow{L^p} f$ such that
\begin{equation}\label{E:Ep_fg_char}
\mathcal{E}_p (f,g)=\lim_{n \to \infty} \frac{1}{r_n^{p}} \int_X \fint_{B(x,r_n)} |f_n(x)-f_n(y)|^{p-2} (f_n(x)-f_n(y))(g_n(x)-g_n(y))d\mu(y) d\mu(x)
\end{equation}
for any $\{g_n\}_{n\geq 1}\subset L^p(X,\mu)$ with $g_n\xrightarrow{L^p} g$.
\end{lem}

\begin{proof}
The proof follows the same lines as before and we start with the case $p \ge 2$. For $x,y \in \mathbb{R}$ and $t \in [-1,1]$,
\begin{equation}\label{E:Taylor_abs_p_02}
|x+ty|^p=|x|^p+p x y  |x|^{p-2} t+t^2R_1(t,x,y),
\end{equation}
where $|R_1(t,x,y)| \le C |y|^2(|x|^{p-2}+|y|^{p-2})$. Consider now $f,g \in KS^{1,p}(X)$. By definition of $\Gamma$-limit, there exists $\{f_n\}_{n\geq 1}\subset L^p(X,\mu)$ such that 
\begin{equation}\label{E:Ep_fg_char_01}
    \mathcal{E}_p(f)=\lim_{n\to\infty}E_{p,r_n}(f_n)
\end{equation}
and for any sequence $\{g_n\}_{n\geq 1}$ with $g_n\xrightarrow{L^p}g$ it holds that
\begin{equation}\label{E:Ep_fg_char_02}
    \mathcal{E}_p(g)\leq\liminf_{n\to\infty}E_{p,r_n}(g_n).
\end{equation}
Moreover, $f_n+tg_n\xrightarrow{L^p}f+tg$, whence~\eqref{E:Ep_fg_char_02} also holds for $f+tg$. Together with~\eqref{E:Taylor_abs_p_02} that yields
\begin{align}
\mathcal{E}_p (f+tg) & \le \liminf_{n \to +\infty} E_{p,r_n}(f_n+t g_n) \notag\\
&=\liminf_{n \to +\infty} \big(E_{p,r_n}(f_n)+tpE_{p,r_n}(f_n,g_n)+t^2R_n(t)\big),\label{E:Ep_fg_char_03}
\end{align}
where
\[
E_{p,r_n}(f_n,g_n):= \frac{1}{r_n^{p\alpha_p}}\int_X \fint_{B(x,r_n)}|f_n(y)-f_n(x)|^{p-2} (f_n(x)-f_n(y))(g_n(x)-g_n(y))d\mu(y)d\mu(x)
\]
and $R_n(t)$ has the property that $\sup_n |R_n(t)|$ is bounded in $t$. Therefore, it follows from~\eqref{E:Ep_fg_char_03} and~\eqref{E:Ep_fg_char_01} that
\[
\mathcal{E}_p(f+tg) \le \mathcal{E}_p(f)+p t \liminf_{n \to +\infty}E_{p,r_n}(f_n,g_n)+t^2\sup_n |R_n(t)|
\]
and thus
\begin{equation}\label{E:Ep_fg_char_04}
\frac{1}{p} \limsup_{t \to 0} \frac{\mathcal{E}_p (f+tg) - \mathcal{E}_p (f)}{t} \le \liminf_{n \to +\infty}E_{p,r_n} (f_n,g_n). 
\end{equation}
Substituting $g$ by $-g$ in the above yields
\[
\frac{1}{p} \limsup_{t \to 0^+} \frac{\mathcal{E}_p (f-tg) - \mathcal{E}_p (f)}{t} \le \liminf_{n \to +\infty}E_{p,r_n} (f_n,-g_n) 
\]
which together with Lemma~\ref{L:same_lim_RL} and~\eqref{E:Ep_fg_char_04} implies
\begin{align*}
    \limsup_{n\to\infty}E_{p,r_n}(f_n,g_n)
    &=\frac{1}{p}\liminf_{n\to\infty}\limsup_{t\to 0^+}\frac{\mathcal{E}_p(f)-\mathcal{E}_p(f-tg)}{t}\\
    &=\frac{1}{p}\liminf_{n\to\infty}\limsup_{t\to 0^+}\frac{\mathcal{E}_p(f+tg)-\mathcal{E}_p(f)}{t}\\
    &\leq\frac{1}{p}\limsup_{t\to 0^+}\frac{\mathcal{E}_p(f+tg)-\mathcal{E}_p(f)}{t}\\
    &\leq \liminf_{n\to\infty}E_{p,r_n}(f_n,g_n).
\end{align*}
Therefore, the limit $\lim_{t \to 0} \frac{\mathcal{E}_p (f+tg)-\mathcal{E}_p (f)}{t}$ exists and is given by~\eqref{E:Ep_fg_char} as required. The proof for $1<p<2$ is similar after using that for $x,y \in \mathbb{R}$ and $t \in [-1,1]$ we have
\[
|x+ty|^p=|x|^p+p y\,{\rm sign}(x) |x|^{p-1} t+t^{2}R_1(t,x,y),
\]
where the Taylor remainder $R_1(t,x,y)$ satisfies $|R_1(t,x,y)| \le C |y|^p$. 
\end{proof}

\begin{remark}\label{R:Non-linearity}
   From the limit expression~\eqref{E:Ep_fg_char} it follows that $\mathcal{E}_p(f,g)$ is not symmetric, it is linear in the second component and not in the first. However, it does satisfy 
   \[
   \mathcal{E}_p(f,f)=\mathcal{E}_p(f)
   \]
   because $\mathcal{E}_p(f+tf)=(1+t)^p \mathcal{E}_p(f)$.
\end{remark} 

\begin{cor}
Assume $p>1$. The form $(\mathcal{E}_p,KS^{1,p}(X))$ is local, that is $\mathcal{E}_p(f,g)=0$ if $f$ or $g$ are constant.
\end{cor}


\subsection{$p$-Laplacian}\label{p laplace}

Throughout this section we assume $p>1$ and use the p-energy $\mathcal{E}_p$ to define an operator acting as $p$-Laplacian as follows: For $f \in KS^{1,p}(X)$ we say that $f$ is in the domain $\mathrm{dom} (\Delta_p)$ of the $p$-Laplacian $\Delta_p$ if there exists $h \in L^q(X,\mu)$, where $\frac{1}{p}+\frac{1}{q}=1$, such that for every $g \in KS^{1,p}(X)$,
\[
\mathcal{E}_p(f,g)=-\int_X h g\, d\mu.
\]
If such a function $h$ exists, it is necessarily unique and we define $\Delta_p f:=h$. As a consequence of Lemma~\ref{L:def_Ep_fg}, one obtains the following approximation of the $p$-Laplacian.

\begin{cor}
Let $p>1$, $f \in KS^{1,p}(X)$ and $\{f_n\}_{n\geq 1}\subset L^p(X,\mu)$ be a sequence such that $f_n\xrightarrow{L^p} f$ and $\lim_{n\to\infty}E_{p,r_n}(f_n)=\mathcal{E}_p (f)$. If the sequence of functions
\[
\Delta_p^n f_n(x):=- \int_{B(x,r_n)} \left(\frac{1}{\mu(B(x,r_n))} + \frac{1}{\mu(B(y,r_n))}\right) | f_n(x)-f_n(y)|^{p-2} (f_n(x)-f_n(y)) d\mu(y)
\]
is uniformly bounded in $L^q(X,\mu)$ with $\frac{1}{p}+\frac{1}{q}=1$, then $f \in \mathrm{dom} (\Delta_p)$ and $\Delta_p^n f_n$ converges weakly to $\Delta_p f$ in $L^q(X,\mu)$.
\end{cor}

\begin{proof}
Since the sequence $\{\Delta_p^n f_n\}_{n\geq 1}$ is uniformly bounded in $L^q(X,\mu)$ by assumption, one can find a subsequence $\{\Delta_p^{n_k} f_{n_k}\}_{k\geq 1}$ that converges weakly to some $h\in L^q(X,\mu)$. Therefore, for every $g \in KS^{1,p}(X)$
\[
\lim_{k\to\infty}\int_X (\Delta_p^{n_k} f_{n_k} ) g d\mu = \int_X h g d\mu.
\]
Further, the definition of $\Delta_p^n$ implies that
\[
\int_X (\Delta_p^{n_k} f_{n_k} ) g d\mu=-\frac{1}{r_{n_k}^{p}} \int_X \fint_{B(x,r_{n_k})} |f_{n_k}(x)-f_{n_k}(y)|^{p-2} (f_{n_k}(x)-f_{n_k}(y))(g(x)-g(y))d\mu(y) d\mu(x),
\]
and Lemma~\ref{L:def_Ep_fg} thus yields
\[
\int_X (\Delta_p^{n_k} f_{n_k} ) g\, d\mu \to -\mathcal{E}_p(f,g).
\]
Hence, $f \in \mathrm{dom} (\Delta_p)$ and $h=\Delta_p f$. By compactness, $\Delta_p^n f_n$ converges weakly to $\Delta_p f$ in $L^q(X,\mu)$.
\end{proof}

\section{Korevaar-Schoen $p$-energy measures}\label{S:KS_pmeasures}
The aim of this section is to associate the $p$-energy introduced in Section~\ref{S:KS_pforms} with a Radon measure in such a way that, for each $f\in KS^{1,p}(X)$, the quantity $\mathcal{E}_p(f)$ may be viewed as the measure of the whole space $X$.

\subsection{Construction of the $p$-energy measures}
To apply the localization method explained in Section~\ref{SS:Gamma_convergence} we start by considering the localized energy functionals $E_{p,r}\colon L^p(X,\mu)\times \mathcal{O}\to \mathbb{R}$ given by
\begin{equation}\label{E:def_Epr_loc}
E_{p,r}(f,U):= \frac{1}{r^{p}}\int_U\fint_{B(x,r)} |f(y)-f(x)|^p d\mu(y) d\mu(x)
\end{equation}
where  $U\subset X$ is a an open set and as before $1\le p<\infty$, $r>0$,.

\begin{remark}\label{R:E_r_superad}
For fixed $r>0$ and $f\in L^p(X,\mu)$, the functional $E_{p,r}(f,\cdot)\colon \mathcal{O}\to[0,\infty]$ is a measure and in particular
\begin{enumerate}[wide=0em,label={\rm (\roman*)}]
\item $E_{p,r}(f,\cdot)$ is superadditive, 
\item $E_{p,r}(f,U)<\infty$ for any $U\subseteq X$ and $f\in KS^{1,p}(X)$.
\end{enumerate}
\end{remark}

As pointed out in Remark~\ref{R:basic_props_gamma_bar}, any sequence of functionals in $L^p(X,\mu)$ has a $\overline{\Gamma}$-convergent subsequence whose limit becomes the natural candidate for a $p$-measure.

\begin{defn}\label{D:def_pre_Gamma_p}
The $p$-energy functional $\Gamma_p\colon L^p(X,\mu)\times\mathcal{O}\to[0,\infty]$ is defined as
\begin{equation}\label{E:def_pre_Gamma_p}
    \Gamma_p:=\overline{\Gamma}{\textbf{-}}\lim_{n\to\infty}E_{p,r_n},
\end{equation}
where $\{E_{p,r_n}\}_{n\geq 1}$ is a $\overline{\Gamma}$-convergent subsequence.
\end{defn}

\begin{remark}
The subsequence $r_n$ in Definition~\ref{D:def_pre_Gamma_p} is a subsequence of the subsequence defined in Theorem~\ref{T:KS_as_Gamma} but we still denote it by $r_n$ to ease the notation.
\end{remark}

The $p$-energy measure associated with a function $f\in KS^{1,p}(X)$ will be denoted as $\Gamma_p(f):=\Gamma_p(f,\cdot)$. 
 
It follows readily from Definition \ref{D:def_Gamma_bar} that  for $f \in KS^{1,p}(X)$ and $U \subseteq X$
\[
\Gamma_p(f)(U) \le \liminf_{n\to +\infty} E_{p,r_n}(f_n,U) \le \mathcal{E}_p(f),
\]
where the sequence $f_n$ above is such that $f_n \xrightarrow{L^p}f$ and $\lim_{n\to +\infty} E_{p,r_n}(f_n,X)=\mathcal{E}_p(f)$.
 We now investigate further properties towards proving in Theorem~\ref{T:KS_as_Gamma} that $\Gamma_p$ is a Radon measure with $\Gamma_p(f)(X)=\mathcal{E}_p(f)$.
 
\begin{lem}\label{L:Gamma_locality}
The functional $\Gamma_p$ is local, that is for any $U\in\mathcal{O}$,
\begin{equation}\label{E:def_local_Gamma}
    \Gamma_p(f)(U)=\Gamma_p(g)(U)
\end{equation}
for all $f,g\in KS^{1,p}(X)$ with $f|_U=g|_U$ $\mu$-a.e.
\end{lem}
\begin{proof}
For a set $\Omega$ we recall the notation
\[
\Omega_r=\left\{ x \in X, d(x,\Omega) \le r  \right\}.
\]
Let $U\in\mathcal{O}$ and $f,g\in KS^{1,p}(X)$ be such that $f|_U=g|_U$ $\mu$-a.e. Further, let $A \Subset U$ and let $\{f_n\}_{n\geq 1}\subset L^p(X,\mu)$ be a sequence that converges to $f$ in $L^p(X,\mu)$ and
\[
 \limsup_{n \to +\infty} E_{p,r_n} (f_n)(A) \le  \Gamma_p(f)(U).
\]
For each $n\geq 1$ define the function $\hat{f}_n$ by
\begin{align*}
\hat{f}_n(x):=
\begin{cases}
\hat{f}_n(x)=f_n(x), \quad x \in U \\
\hat{f}_n(x)=g(x), \quad x \notin U.
\end{cases}
\end{align*}
Since $f_n\xrightarrow{L^p}f$ and $f=g$ a.e. on $U$, it follows that $\hat{f}_n\xrightarrow{L^p}g$. Therefore, by $\overline{\Gamma}$-convergence,
\[
\Gamma_p(g)(A) \le \liminf_{n \to +\infty} E_p (\hat{f}_n,r_n)(A).
\]
Further, since $ A \Subset U$, it holds that $A_{r_n} \subset U$ for $n$ large enough, whence $E_{p,r_n}(\hat{f}_n,A)=E_{p,r_n} (f_n,A)$ and thus
\[
\Gamma_p(g)(A) \le \liminf_{n \to +\infty} E_{p,r_n}(f_n,A) \le \Gamma_p(f)(U).
\]
Since the above holds for all $A \Subset U$, we deduce from the inner regularity of $\Gamma_p(g)$ that $\Gamma_p(g)(U) \le \Gamma_p(f)(U)$. Similar arguments show that $\Gamma_p(f)(U) \le \Gamma_p(g)(U)$.
\end{proof}

To prove that $\Gamma_p(f)$ in fact defines a Borel measure on $X$ in the next theorem, we rely on a particular characterization of measures that can be found in~\cite[Theorem 18.5]{DMas93}. One of the main ingredients to establish the result is a variation of what dal Maso calls the ``fundamental estimate'' in~\cite[Definition 18.2]{DMas93}, see Lemma~\ref{L:fundamental_estimate}.  The fact that the underlying space is $\sigma$-compact will allow to conclude that the measure is Radon, i.e. it is finite on compact sets, outer regular on Borel sets, and inner regular on open sets.


\begin{thm}\label{T:Gamma_measure}
Let $1<p<\infty$. For every $f\in KS^{1,p}(X)$, $\Gamma_{p}(f)$ defines a finite Radon measure on $X$ such that
\[
\Gamma_{p}(f)(X)=\mathcal{E}_p(f).
\]
Moreover, there exist $C_1,C_2>0$ such that for any $f\in KS^{1,p}(X)$ and any $U,V,W \in \mathcal{O}$ such that $U \Subset V \Subset W$,
    \begin{equation}\label{E:local_KS_comp_Gamma}
        C_1\limsup_{r\to 0}E_{p,r}(f,U)\leq \Gamma_p(f)(V)\leq \liminf_{n \to +\infty}E_{p,r_n}(f,V) \le   C_2\liminf_{r\to 0^+}E_{p,r}(f,W).
    \end{equation}
\end{thm}

\begin{remark}
As mentioned in Remark~\ref{R:KS_equiv_weakUG}, for $p=1$ the space $KS^{1,1}(X)$ coincides with the space of bounded variation functions $BV(X)$. The Radon measure $\Gamma_{1}(f)$ should then be interpreted as the BV measure of $f$ associated with the total variation $\mathcal E_1(f)$. We refer to Theorem~\ref{T:BV_case} for a comparison with the BV measures introduced by Miranda in~\cite{miranda}.
\end{remark}

\begin{proof}[Proof of~\ref{T:Gamma_measure}]
By virtue of ~\cite[Theorem 5.1]{GiorgiLetta}, see also~\cite[Theorem 5.1]{DMas93}, $\Gamma_{p}(f)$ defines a Borel measure if and only if it is subadditive, superadditive and inner regular. Subadditivity is proved in Proposition~\ref{P:Gamma_subadditive}, while inner regularity follows from the definition of $\overline{\Gamma}$-convergence, cf. Remark~\ref{R:basic_props_gamma_bar}. Further, since $E_{p,r_n}(f,\cdot)$ is a measure for any $f\in KS^{1,p}(X)$ and $r_n>0$, it is superadditive and thus its $\overline{\Gamma}$-limit also is, c.f.~\cite[Proposition 16.12]{DMas93}.

\medskip

Note also that for any compact $K\subset X$ and $f\in KS^{1,p}(X)$ we have $\Gamma_{p}(f)(K)<\infty$. Since the underlying space $X$ is complete and $\sigma$-compact, 
it follows from~\cite[Theorem 2.18]{Rud87} that $\Gamma_p(f)$ is (in particular outer) regular.

\medskip

Consider now $f\in KS^{1,p}(X)$ and $U,V,W \in \mathcal{O}$ with $U \Subset V \Subset W$. From the characterization of $\overline{\Gamma}$-limits in Definition~\ref{D:def_Gamma_bar}, choosing the trivial sequence $f_n=f$ and the set $V\in\mathcal{O}$, the second and third inequality in~\eqref{E:local_KS_comp_Gamma} follow from Lemma ~\ref{L:local_Kumagai-Sturm_cond}.
To prove the first inequality, recall also from Definition~\ref{D:def_Gamma_bar} that there exists $\{f_n\}_{n\geq 1}$ converging strongly to $f$ in $L^p(X,\mu)$ with
\begin{equation}
\limsup_{n\to\infty}E_{p,r_n}(f_n,U') \leq \Gamma_p(f)(V),
\end{equation}
where $U'\in\mathcal{O}$ is such that $U \Subset U' \Subset V$. Applying Lemma~\ref{L:local_Kumagai-Sturm_cond} to that sequence and $U$, the desired inequality follows since for $r>0$ small enough $U_r \subset U'$.
\end{proof}

The next lemma corresponds to dal Maso's ``fundamental estimate'', which holds uniformly for the subsequence defining $\Gamma_p$ in~\eqref{E:def_pre_Gamma_p}.

\begin{lem}\label{L:fundamental_estimate}
For any $A,A',B\in\mathcal{O}$ with $A'\Subset A$ there exists a continuous cutoff function $\varphi$ with $0\leq\varphi\leq 1$, $\varphi|_{A'}\equiv 1$ and $\supp\varphi\subset A$ such that 
    \begin{align}\label{E:fundamental_estimate}
    E_{p,r}(\varphi f+(1-\varphi)g)(A'\cup B)&\leq (1-\varepsilon)^{1-p}  \big(E_{p,r}(f)(A)+E_{p,r}(g)(B)\big)+C \varepsilon^{1-p}\int_{S_r}|f-g|^pd\mu,
\end{align}
where
\[
S_r:=(A'\cup B)_r\cap  (A\setminus A')_{3r},
\]
for any $0<\varepsilon <1$, $r >0$ and $f,g\in L^p(X,\mu)$. 
The constant $C>0$ above depends only on $A,A'$ and the doubling constant of $X$.
\end{lem}

\begin{remark}
The estimate in Lemma~\ref{L:fundamental_estimate} is slightly stronger than the original fundamental estimate by dal Maso in~\cite[Definition 18.2]{DMas93}. The latter only requires that for any $\varepsilon>0$ and any $A,A',B\in\mathcal{O}$ with $A'\Subset A$ there exists $C>0$ with the property that for every $f,g\in L^p(X,\mu)$ there is a function $\varphi\in{\rm cutoff}(A,A')$ for which~\eqref{E:fundamental_estimate} holds. In particular, $\varphi$ \emph{may} depend of $f,g$, while it does not in Lemma~\ref{L:fundamental_estimate}.
\end{remark}

\begin{proof}[Proof of Lemma~\ref{L:fundamental_estimate}]
Let $A,A',B\in\mathcal{O}$ with $A'\Subset A$ and let $0<\varepsilon<1$. For $x \in X$ consider
\[
\varphi (x):=\frac{\min \left\{ d(x,A^c), d(A',A^c) \right\} }{d(A',A^c)}.
\]
Note that $0\leq\varphi\leq 1$, $\varphi|_{A'}\equiv 1$ and $\supp\varphi\subset A$. Then,
\begin{align*}
 &   \int_{A'\cup B} \fint_{B(x,r)} | [ \varphi (y) f(y)+(1-\varphi (y))g(y)] -[\varphi (x) f(x)+(1-\varphi (x))g(x)]|^p d\mu(y) d\mu(x) \\
=   & \int_{A'\cup B} \fint_{B(x,r)} |\varphi (x) (f(x)-f(y)) +(1-\varphi(x))(g(x)-g(y)) +(\varphi(x)-\varphi(y))( f(y)-g(y))  |^p d\mu(y) d\mu(x).
\end{align*}
Applying the convexity inequality, it follows that
\[
(a+b)^p \le (1-\varepsilon)^{1-p} a^p+ \varepsilon^{1-p} b^p
\]
that is valid for $a,b \ge 0$ and $0< \varepsilon <1$ yields
\begin{align*}
 &   \int_{A'\cup B} \fint_{B(x,r)} | [ \varphi (y) f(y)+(1-\varphi (y))g(y)] -[\varphi (x) f(x)+(1-\varphi (x))g(x)]|^p d\mu(y) d\mu(x) \\
\le  &  (1-\varepsilon)^{1-p}  \int_{A'\cup B} \fint_{B(x,r)} |\varphi (x) (f(x)-f(y)) +(1-\varphi(x))(g(x)-g(y))  |^p d\mu(y) d\mu(x) \\
 &+ \varepsilon^{1-p}  \int_{A'\cup B} \fint_{B(x,r)} |(\varphi(x)-\varphi(y))( f(y)-g(y))  |^p d\mu(y) d\mu(x).
\end{align*}
The first term can be bounded again by convexity since
\begin{align*}
&  \int_{A'\cup B} \fint_{B(x,r)} |\varphi (x) (f(x)-f(y)) +(1-\varphi(x))(g(x)-g(y))  |^p d\mu(y) d\mu(x) \\
\le & \int_{A' \cup B} \fint_{B(x,r)} \left( \varphi (x) | f(x)-f(y)   |^p + (1-\varphi(x))  | g(x)-g(y)   |^p  \right)d\mu(y) d\mu(x) \\
\le & \int_{A} \fint_{B(x,r)}  | f(x)-f(y)   |^p d\mu(y) d\mu(x) + \int_{B} \fint_{B(x,r)}  | g(x)-g(y)   |^p d\mu(y) d\mu(x) .
\end{align*}
To bound the second term, we observe that for $x,y \in X$ with $d(x,y) \le r$, one has $\varphi(x)=\varphi(y)$ if $x \notin (A\setminus A')_{2r}$. 
The Lipschitz property of $\varphi$, Fubini theorem, and the volume doubling property finally imply
\begin{align*}
& \int_{A'\cup B} \fint_{B(x,r)} |(\varphi(x)-\varphi(y))( f(y)-g(y))  |^p d\mu(y) d\mu(x) \\
\le & C r^p \int_{(A'\cup B)\cap  (A\setminus A')_{2r} } \fint_{B(x,r)} | f(y)-g(y)  |^p d\mu(y) d\mu(x)   \\
\le & C r^p \int_{(A'\cup B)_r\cap  (A\setminus A')_{3r} } \int_{B(y,r)}\frac{ d\mu(x)}{\mu(B(x,r))}  |f(y)-g(y)  |^p d\mu(y)  \\
\le & C r^p  \int_{(A'\cup B)_r\cap  (A\setminus A')_{3r} }  |f(y)-g(y)  |^p d\mu(y).
\end{align*}
\end{proof}

The next lemma records a consequent estimate that will be used to prove the subadditivity of $\Gamma_p(f)$ in Proposition~\ref{P:Gamma_subadditive}.

\begin{lem}\label{L:est4subad}
For any $A',A'',B\in\mathcal{O}$ with $A'\Subset A''$ and $f\in KS^{1,p}(X)$,
\begin{equation*}
\Gamma{\text{-}}\limsup_{n \to +\infty} E_{p,r_n} (f,A'\cup B)\leq \Gamma{\text{-}}\limsup_{n \to +\infty} E_{p,r_n}(f,A'')+\Gamma {\text{-}}\limsup_{n \to +\infty} E_{p,r_n} (f,B).
\end{equation*}
\end{lem}
\begin{proof}
The characterization of $\Gamma {\text{-}}\limsup$ in~\cite[Proposition 8.1]{DMas93}, see also Definition~\ref{D:Gamma_limit}, provides for $A'',B\in\mathcal{O}$ as required the existence of sequences $\{g_n\}_{n\geq 1}$ and $\{h_n\}_{n\geq 1}$ in $L^p(X,\mu)$ such that both $g_n,h_n\xrightarrow{L^p}f$ and 
\begin{equation}\label{E:est4subad_01}
 \begin{aligned}
&\Gamma {\text{-}}\limsup E_{p,r_n}(f,A'')=\limsup_{n\to\infty}E_{p,r_n}(g_n,A'')\\
&\Gamma {\text{-}}\limsup E_{p,r_n}(f,B)=\limsup_{n\to\infty}E_{p,r_n}(h_n,B).
\end{aligned}
\end{equation}
Let $0<\varepsilon <1$ be arbitrary but fixed. By virtue of the fundamental estimate~\eqref{E:fundamental_estimate} there exists $\varphi\in{\rm cutoff}(A,A')$ such that
\begin{equation*}
E_{p,r_n}(\varphi g_n+(1-\varphi)h_n,A'\cup B)\leq (1-\varepsilon)^{1-p}\big[E_{p,r_n}(g_n,A'')+E_{r_n}(h_n,B)\big]+M\varepsilon^{1-p}\|g_n-h_n\|_{L^p(S_{r_n},\mu)}^p.
\end{equation*}
Since $\varphi g_n+(1-\varphi)h_n\xrightarrow{L^p} f$, the latter estimate together with the definition of $\Gamma$-limit and~\eqref{E:est4subad_01} yields
\begin{align*}
   \Gamma {\text{-}}\limsup_{n\to\infty}E_{p,r_n}(f,A'\cup B')
    &\leq   \Gamma {\text{-}}\limsup_{n\to\infty} E_{p,r_n}(\varphi g_n+(1-\varphi)h_n,A'\cup B) \\
    &\leq  (1-\varepsilon)^{1-p} \left( \Gamma {\text{-}}\limsup_{n\to\infty}E_{p,r_n}(f,A'')+ \Gamma {\text{-}}\limsup_{n\to\infty}E_{p,r_n}(f,B) \right).
\end{align*}
By letting $\varepsilon \to 0$ the assertion of the lemma follows.
\end{proof}

\begin{prop}\label{P:Gamma_subadditive}
For any $f\in KS^{1,p}(X)$, the functional $\Gamma_p(f)$ is subadditive, i.e. for any $A,B\in\mathcal{O}$,
\begin{equation*}
\Gamma_p(f)(A\cup B)\leq \Gamma_p(f)(A)+\Gamma_p(f)(B).
\end{equation*}
\end{prop}
\begin{proof}
We follow~\cite[Proposition 18.4]{DMas93}. Let $f \in KS^{1,p}(X)$ and $A,B \in \mathcal O$. By definition of $\overline{\Gamma}$-convergence, $\Gamma_p$ is in fact the inner regular envelope of the functional $\Gamma-\limsup_{n\to +\infty}E_{p,r_n}$, see e.g.~\cite[Definition 16.2]{DMas93}. Thus, for any $0<s<\Gamma_p(f)(A\cup B)$, there exists $C\in\mathcal{O}$ with the properties that $C\Subset A\cup B$ and 
\begin{equation*}
s<\Gamma {\text{-}}\limsup_{n\to\infty}E_{p,r_n}(f,C).
\end{equation*}
Since $C\Subset A\cup B$, there exist sets $A',A'',B'\in\mathcal{O}$ with $A'\Subset A''\Subset A$, $B'\Subset B$ and $C\Subset A'\cup B'$, whence from Lemma \ref{L:est4subad}
\begin{equation*}
    \Gamma {\text{-}}\limsup_{n\to\infty}E_{p,r_n}(f,C)\leq \Gamma {\text{-}}\limsup_{n\to\infty}E_{p,r_n}(f,A'\cup B')\le \Gamma {\text{-}}\limsup_{n\to\infty}E_{p,r_n}(f,A'')+\Gamma {\text{-}}\limsup_{n\to\infty}E_{p,r_n}(f,B').
\end{equation*}
Since $A''\Subset A$ and $B'\Subset B$ we have
\[
\Gamma {\text{-}}\limsup_{n\to\infty}E_{p,r_n}(f,A'') \le \Gamma_p(f)(A), \quad \Gamma {\text{-}}\limsup_{n\to\infty}E_{p,r_n}(f,B') \le \Gamma_p(f)(B).
\]
Therefore we obtain
\[
s \le \Gamma_p(f)(A) +\Gamma_p(f)(B)
\]
and the conclusion follows by letting $s$ converge to $\Gamma_p(f)(A\cup B)$.
\end{proof}

\subsection{Properties of $p$-energy measures}
This section collects several desirable properties for a $p$-energy measure which are extensions of those corresponding to the case $p=2$, see for instance~\cite[Section 3.2]{FOT11} and~\cite{Hin10}.

\begin{prop}\label{P:Minkowski}
For any $f,g\in KS^{1,p}(X)$, $a,b\in\mathbb{R}$ and $U\in\mathcal{O}$,
\begin{equation}\label{E:Minkowski_full}
    \Gamma_p(af+bg)(U)^{1/p}\leq |a|\Gamma_p(f)(U)^{1/p}+|b|\Gamma_p(g)(U)^{1/p}.
\end{equation}
\end{prop}
\begin{proof}
    By the characterization of $\overline{\Gamma}$-convergence, there exist sequences $\{f_n\}_{n\geq 1}$ and $\{g_n\}_{n\geq 1}$ such that
    \begin{equation}\label{E:Mink_01}
    \begin{aligned}
        &\limsup_{n\to\infty}E_{p,r_n}(f_n,U')\leq \Gamma_p(f)(U)\\
        &\limsup_{n\to\infty}E_{p,r_n}(g_n,U')\leq \Gamma_p(g)(U)
    \end{aligned}
    \end{equation}
    for $U'\Subset U$. Further, since $af_n+bg_n\xrightarrow{L^p}af+bg$, it also holds that
    \begin{equation}\label{E:Mink_02}
    \Gamma_p(af+bg)(U')\leq\liminf_{n\to\infty}E_{p,r_n}(af_n+bg_n,U').
    \end{equation}
    By virtue of Minkowski's inequality,
    \begin{align*}
        E_{p,r_n}(af_n+bg_n,U')^{1/p}&=\bigg(\frac{1}{r_n^p}\int_{U'}\fint_{B(x,r_n)}|af_n(x)-f_n(y)+bg_n(x)-g_n(y)|^pd\mu(y)\,d\mu(x)\bigg)^{1/p}\\
        &\leq |a|E_{p,r_n}(f_n,U')^{1/p}+|b|E_{p,r_n}(g_n,U')^{1/p}.
    \end{align*}
    Taking $\limsup_{n\to\infty}$ on both sides of the inequality and applying~\eqref{E:Mink_01} and~\eqref{E:Mink_02} we arrive at
    \begin{equation}\label{E:Minkowski_part}
     \Gamma_p(af+bg)(U')^{1/p}\leq |a|\Gamma_p(f)(U)^{1/p}+|b|\Gamma_p(g)(U)^{1/p}.
    \end{equation}
    Finally, due to Theorem~\ref{T:Gamma_measure} we know that $\Gamma_p$ is Radon, whence~\eqref{E:Minkowski_part} in particular implies that
    \begin{align*}
    \Gamma_p(af+bg)(U)^{1/p}&=(\sup\{\Gamma_p(af+bg)(U')\colon U'\Subset U,\,U\text{ compact}\})^{1/p}\\
    &\leq |a|\Gamma_p(f)(U)^{1/p}+|b|\Gamma_p(g)(U)^{1/p}.
    \end{align*}
\end{proof}
The next properties will be especially relevant to prove the absolute continuity of $\Gamma_p(f)$ with respect to the underlying measure $\mu$ in Section~\ref{SS:absolute_cont}.

\begin{lem}\label{L:lin_abs_cont}
Let $f,g\in KS^{1,p}(X)$ and $a,b\in \mathbb{R}$. If $\Gamma_p(f)$ and $\Gamma_p(g)$ are absolutely continuous with respect to $\mu$, then also $\Gamma_p(af+bg)$ is.
\end{lem}
\begin{proof}
    Let $U\in\mathcal{O}$ be such that $\mu(U)=0$. By assumption, also $\Gamma_p(f)(U)=0=\Gamma_p(g)(U)$, and $\Gamma_p(af)(U)=0=\Gamma_p(bg)(U)$. By virtue of Lemma~\ref{P:Minkowski}, $\Gamma_p(af+bg)(U)=0$ whence $\Gamma_p(af+bg)\ll\mu$.
\end{proof}

\begin{lem}\label{L:sequence_abs_cont}
Let $f\in KS^{1,p}(X)$ and $\{f_n\}_{n\geq 1}\subset KS^{1,p}(X)$ with $\mathcal{E}_p(f-f_n)\xrightarrow{n\to\infty} 0$. If $\Gamma_p(f_n)$ is absolutely continuous with respect to $\mu$, then also $\Gamma_p(f)$ is.
\end{lem}
\begin{proof}
   Let $U\in\mathcal{O}$ be such that $\mu(U)=0$. By assumption, also $\Gamma_p(f_n)(U)=0$ for all $n\geq 1$. By virtue of Proposition~\ref{P:Minkowski} and Theorem~\ref{T:KS_as_Gamma},
   \begin{align*}
       \Gamma_p(f)(U)^{1/p}&=\Gamma_p(f-f_n+f_n)(U)^{1/p}\leq \Gamma_p(f-f_n)(U)^{1/p}+\Gamma_p(f_n)(U)^{1/p}\\
       &=\Gamma_p(f-f_n)(U)^{1/p}\leq \Gamma_p(f-f_n)(X)^{1/p}\\
       &=\mathcal{E}_p(f-f_n)^{1/p}\xrightarrow{n\to\infty} 0.
   \end{align*}
\end{proof}

\subsection{$(p,p)$-Poincar\'e inequality with respect to the $p$-energy measure}\label{SS:pPI}
As it was the case when $p=2$, c.f.~\cite[Theorem 3.4]{ARB23}, the $(p,p)$-Poincar\'e inequality from Assumption~\ref{A:pPI_Lip} that is characteristic of Cheeger spaces, involves the Lipschitz constant of the function. In this section we show that the same equality will hold with the $p$-energy measure on the right hand side instead.

\begin{prop}\label{P:pPI_Gamma_p}
There exists $C>0$ and $\Lambda>1$ such that 
\begin{equation}\label{E:pPI_Gamma_p}
\int_{B(x,R)}|f(y)-f_{B(x,R)}|^pd\mu(z)\leq CR^p\int_{B(x,\Lambda R)}d\Gamma_p(f)
\end{equation}
for any $f\in KS^{1,p}(X)$, $x\in X$ and $R>0$.
\end{prop}
The first part of the proof follows similar arguments as~\cite[Theorem 3.4]{ARB23}, whose details we include for completeness. The second part will make use of some of the properties established previously in this section.

\begin{proof}
\begin{enumerate}[wide=0em,label=\emph{Step} \arabic*:,itemsep=.5em]
    \item Let $f\in {\rm Lip}_{\rm loc}(X)\cap C_c(X)$ and $\varepsilon>0$. By virtue of Proposition~\ref{P:Lip_const_estimate}, the function $f_{\eps}$ as defined in~\eqref{E:def_Lip_approx} is locally Lipschitz and 
\begin{equation}\label{E:Lip_const_bound}
    ({\rm Lip}f_{\varepsilon})^p\leq \frac{C}{\varepsilon^p}\fint_{5B_i^\varepsilon}\fint_{B(z,2\varepsilon)}|f(z)-f(y)|^pd\mu(y)\,d\mu(z)
\end{equation}
on each $B_i^\varepsilon$. 
\item  For $R>0$ and $0 <\varepsilon <R$, it follows from~\eqref{E:Lip_const_bound} that
\begin{align}
    \int_{B(x,R)}({\rm Lip}f_{\varepsilon})^pd\mu&\leq \sum_{i, B_i^\varepsilon \cap B(x,R) \neq \emptyset  } \int_{B_i^\varepsilon}({\rm Lip}f_{\varepsilon})^p(x)d\mu(x)\notag\\
    &\leq C \sum_{i, B_i^\varepsilon \cap B(x,R) \neq \emptyset  }\int_{B_i^\varepsilon}\frac{1}{\varepsilon^p}\fint_{5B_i^\varepsilon}\fint_{B(z,2\varepsilon)}|f(z)-f(y)|^pd\mu(y)\,d\mu(z)\,d\mu(x)\notag\\
    &\leq C\sum_{i, B_i^\varepsilon \cap B(x,R) \neq \emptyset  } \int_{5B_i^\varepsilon}\frac{1}{\varepsilon^p}\fint_{B(z,2\varepsilon)}|f(z)-f(y)|^pd\mu(y)\,d\mu(z)\notag\\
    &\leq C\int_{B(x,7R)}\frac{1}{\varepsilon^p}\fint_{B(z,2\varepsilon)}|f(z)-f(y)|^pd\mu(y)\,d\mu(z).\label{E:pPI_Gamma_02}
\end{align}
By virtue of Proposition~\ref{P:Lip_const_estimate}, we have that $f_{\varepsilon/2}$ converges to $f$ in $L^p(X,\mu)$ as $\varepsilon\to 0^+$.
\item
The convexity of the function $x \to x^p$ inequality now implies
\begin{align}
    \int_{B(x,R)}|f(z)-f_{B(x,R)}|^pd\mu(z)&\leq 3^{p-1}\int_{B(x,R)}|f(z)-f_{\varepsilon/2}(z)|^pd\mu(z)\notag\\
    &+3^{p-1}\int_{B(x,R)}|f_{\varepsilon/2}(z)-(f_{\varepsilon/2})_{B(x,R)}|^pd\mu(z)\notag\\
    &+3^{p-1}\int_{B(x,R)}|(f_{\varepsilon/2})_{B(x,R)}-f_{B(x,R)}|^pd\mu(z).\label{E:pPI_Gamma_03}
\end{align}
\item The first term in~\eqref{E:pPI_Gamma_03} is bounded by $\|f-f_{\varepsilon/2}\|_{L^p(X,\mu)}^p$ and the third also using Cauchy-Schwarz inequality because
\begin{align*}
 \int_{B(x,R)}|(f_{\varepsilon/2})_{B(x,R)}-f_{B(x,R)}|^pd\mu(z) 
 &= \mu(B(x,R))\bigg|\fint_{B(x,R)}(f_{\varepsilon/2}(y)-f(y))\,d\mu(y)\bigg|^p
 \leq \|f-f_{\varepsilon/2}\|_{L^p(X,\mu)}^p.
\end{align*}

\item
For the second term in~\eqref{E:pPI_Gamma_03}, since $f_{\varepsilon}\in{\rm Lip}_{\rm loc}(X)\cap C_0(X)$, the $(p,p)$-Poincar\'e inequality~\eqref{E:pPI_Lip} and~\eqref{E:pPI_Gamma_02} imply
\begin{align*}
    \int_{B(x,R)}|f_{\varepsilon/2}(z)-(f_{\varepsilon/2})_{B(x,R)}|^pd\mu(z)
    &\leq CR^p\int_{B(x,\lambda R)}({\rm Lip}f_{\varepsilon/2})^pd\mu\notag\\
    &\leq CR^p\int_{B(x,7\lambda R)}\frac{1}{\varepsilon^p}\int_{B(x,\varepsilon)}|f(z)-f(y)|^pd\mu(y)\,d\mu(z)\\
    &=CR^pE_{p,\varepsilon}(f,B(x,7\lambda R)).
\end{align*}

\item Combining the last two steps with~\eqref{E:pPI_Gamma_03} yields
\begin{equation*}
 \int_{B(x,R)}|f(z)-f_{B(x,R)}|^pd\mu(z)
 \leq C \|f-f_{\varepsilon/2}\|_{L^p(X,\mu)}^p+ CR^pE_{p,\varepsilon}(f,B(x,7\lambda R)).
\end{equation*}
Taking $\limsup_{\varepsilon>0}$ on both sides of the inequality above, it follows from~\eqref{E:local_KS_comp_Gamma} in Theorem~\ref{T:Gamma_measure} that
\begin{equation}\label{E:pPI_Gamma_04}
     \int_{B(x,R)}|f(z)-f_{B(x,R)}|^2d\mu(z)\leq CR^p\int_{B(x,\lambda'R)}d\Gamma_p (f)
\end{equation}
for some $\lambda'>1$ independent of $R$ and $f$.
\item Let now $f\in KS^{1,p}(X)$. In view of Corollary~\ref{C:Lip_density}, there is a sequence $\{f_n\}_{n\geq 0} \subset {\rm Lip}_{\rm loc}(X)\cap C_c(X)$ that converges to $f$ with respect to $(\mathcal{E}_p(\cdot,\cdot)+\|\cdot\|_{L^p(X,\mu)})^{1/p}$. Applying again the basic convexity inequality and~\eqref{E:pPI_Gamma_04} yields
\begin{align}
    \int_{B(x,R)}|f(z)-f_{B(x,R)}|^pd\mu(z)&\leq 3^{p-1}\int_{B(x,R)}|f(z)-f_n(z)|^pd\mu(z)\notag\\
    &+3^{p-1}\int_{B(x,R)}|f_n(z)-(f_n)_{B(x,R)}|^pd\mu(z)\notag\\
    &+3^{p-1}\int_{B(x,R)}|(f_n)_{B(x,R)}-f_{B(x,R)}|^pd\mu(z)\notag\\
    &\leq C \|f-f_n\|_{L^p(X,\mu)}^p
    + CR^p\int_{B(x,\Lambda'R)}d\Gamma_p (f_n)\label{E:pPI_Gamma_05}
\end{align}
with $\Lambda'>1$ possibly different than $\lambda'$. 
\item Note now that Proposition~\ref{P:Gamma_subadditive} implies
\begin{align*}
    \Gamma_p (f_n)(B(x,\Lambda'R))^{1/p}&\leq \Gamma_p(f_n-f)(B(x,\Lambda'R))^{1/p}+\Gamma_p(f)(B(x,\Lambda'R))^{1/p}\\
    &\leq \mathcal{E}_p(f_n-f)^{1/p}+\Gamma_p(f)(B(x,\Lambda'R))^{1/p},
\end{align*}
which combined with~\eqref{E:pPI_Gamma_05} yields
\begin{equation*}
    \int_{B(x,R)}|f(z)-f_{B(x,R)}|^pd\mu(z)\leq C \|f-f_n\|_{L^p(X,\mu)}^p+CR^p\mathcal{E}_p(f_n-f)+CR^p\Gamma_p(f)(B(x,\Lambda'R)).
\end{equation*}
Letting $n\to\infty$ on the right hand side above we finally obtain~\eqref{E:pPI_Gamma_p}.
\end{enumerate}
\end{proof}

\subsection{Absolute continuity for $p>1$}\label{SS:absolute_cont}
To show the absolute continuity of $\Gamma_p(f)$ with respect to the underlying measure $\mu$ when $p>1$ we combine ideas in~\cite{MS23} and~\cite{ARB23}. The proof again will make use of the approximation by Lipschitz functions discussed in Section~\ref{SS:Observations}. We begin by observing that the $p$-energy associated to each $\varphi_i^\varepsilon$ is absolutely continuous with respect to the underlying measure $\mu$.

\medskip

Throughout the section we assume that $p>1$.

\begin{lem}\label{L:abs_cont_partition}
Let $\{B_i^\varepsilon\}_{i\geq 1}$ be an $\varepsilon$-covering of $X$ and $\{\varphi_i^\varepsilon\}_{i\geq 1}$ its associated partition of unity. There exists $C>0$ such that 
\begin{equation}
    \Gamma_p(\varphi_i^\varepsilon)(U)\leq \frac{C^p}{\varepsilon^p}\mu(U)
\end{equation}
for any $U\in\mathcal{O}$.
\end{lem}
\begin{proof}
    By virtue of~\cite[p.109]{HKST15}, each function $\varphi_i^\varepsilon$ is $(C/\varepsilon)$-Lipschitz and $\supp \varphi_i^\varepsilon\subset B_i^{2\varepsilon}$. Let now $U\in\mathcal{O}$. 
    \begin{align*}
        E_{p,r_n}(\varphi_i^\varepsilon,U)&=\frac{1}{r_n^p}\int_U\fint_{B(x,r_n)}|\varphi_i^\varepsilon(x)-\varphi_i^\varepsilon(y)|^pd\mu(y)\,d\mu(x)\\
        &\leq \frac{C^p}{\varepsilon^p}\mu(U).
    \end{align*}
    Taking $\liminf_{n\to\infty}$ on both sides of the inequality, and using the characterization of $\overline{\Gamma}$-convergence
    \[
    \Gamma_p(\varphi_i^\varepsilon)(U)\leq \liminf_{n\to\infty}E_{p,r_n}(\varphi_i^\varepsilon,U)\leq \frac{C^p}{\varepsilon^p}\mu(U)
    \]
    as we wanted to prove.
\end{proof}

We now extend absolute continuity to any function in $KS^{1,p}(X)$.

\begin{thm}\label{T:abs_cont_Gamma_p}
For any $f\in KS^{1,p}(X)$, the p-energy measure $\Gamma_p(f)$ is absolutely continuous with respect to the underlying measure $\mu$.
\end{thm}
\begin{proof}
\begin{enumerate}[wide=0em,label=\emph{Step} \arabic*:,itemsep=.5em]
\item Let $f\in {\rm Lip}_{\rm loc}(X)\cap C_0(X)$, $\varepsilon>0$ and consider the corresponding Lipschitz approximation $f_\varepsilon$ from~\eqref{E:def_Lip_approx}. For any $i,j\geq 1$, Cauchy-Schwarz and Proposition~\ref{P:for_local_Kumagai_Sturm} with $U=B_i^\varepsilon$ yield
\begin{align}
    |f_{B_i^\varepsilon}-f_{B_j^\varepsilon}|^p&\leq \bigg(\fint_{B_i^\varepsilon}\fint_{B(x,2\varepsilon)}|f(x)-f(y)|^pd\mu(y)\,d\mu(x)\bigg)^p\notag\\
    &\leq \fint_{B_i^\varepsilon}\fint_{B(x,2\varepsilon)}|f(x)-f(y)|^pd\mu(y)\,d\mu(x)\notag\\
    &\leq\frac{C\varepsilon^p}{\mu(B_i^\varepsilon)}\int_{5\Lambda B_i^\varepsilon}|{\rm Lip }f|^pd\mu.\label{E:abs_cont_Gamma_p_02}
\end{align}
\item The properties of the partition $\{\varphi_i^\varepsilon\}_{i\geq 1}$, c.f.~\cite[p.109]{HKST15}, also imply 
\begin{align*}
    f_\varepsilon(x)&=f_{B_i^\varepsilon}+f_\varepsilon(x)-f_{B_i^\varepsilon}\\
    &=f_{B_i^\varepsilon}+\sum_{j\geq 1}(f_{B_j^\varepsilon}-f_{B_i^\varepsilon})\varphi_j^\varepsilon(x)\\
    &=f_{B_i^\varepsilon}+\sum_{j:B_i^{2\eps}\cap B_j^{2\eps}\ne\emptyset}(f_{B_j^\varepsilon}-f_{B_i^\varepsilon})\varphi_j^\varepsilon(x)
\end{align*}
for any $i\geq 0$ and $x\in B_i^\varepsilon$.
\item Combining Proposition~\ref{P:Minkowski}, Lemma~\ref{L:abs_cont_partition} and~\eqref{E:abs_cont_Gamma_p_02} yields
\begin{align*}
    \Big(\Gamma_p(f_\varepsilon)(B_i^\varepsilon)\Big)^{1/p}&\leq \Big(\Gamma_p(f_{B_i^\varepsilon})(B_i^\varepsilon)\Big)^{1/p}+\sum_{j:B_i^{2\eps}\cap B_j^{2\eps}\ne\emptyset}\Big(\Gamma_p\big((f_{B_j^\varepsilon}-f_{B_i^\varepsilon})\varphi_j^\varepsilon\big)(B_i^\varepsilon)\Big)^{1/p}\\
    &\leq \sum_{j:B_i^{2\eps}\cap B_j^{2\eps}}|f_{B_j^\varepsilon}-f_{B_i^\varepsilon}|\Big(\Gamma_p(\varphi_j^\varepsilon)(B_i^\varepsilon)\Big)^{1/p}\\
    &\leq \sum_{j:B_i^{2\eps}\cap B_j^{2\eps}}\frac{C\varepsilon}{\mu(B_i^\varepsilon)^{1/p}}\bigg(\int_{\Lambda B_i^\varepsilon}|{\rm Lip }f|^pd\mu\bigg)^{1/p}\frac{C}{\varepsilon}\mu(B_i^\varepsilon)^{1/p}\\
    &\leq C\bigg(\int_{\Lambda B_i^\varepsilon}|{\rm Lip }f|^pd\mu\bigg)^{1/p}.
\end{align*}
\item In view of Theorem~\ref{T:Gamma_measure}, for any $\varepsilon>0$
\begin{equation}\label{E:abs_cont_Gamma_p_03}
    \mathcal{E}_p(f_\varepsilon)=\Gamma_p(f_\varepsilon)(X)\leq \sum_{i\geq 1}\Gamma_p(f_\varepsilon)(B_i^\varepsilon)
    \leq C\sum_{i\geq 1}\int_{\Lambda B_i^\varepsilon}|{\rm Lip }f|^pd\mu\leq C\int_X|{\rm Lip }f|^pd\mu,
\end{equation}
whence any sequence $\{f_{\varepsilon_n}\}_{n\geq 1}$ with $\varepsilon_n\to 0$ is uniformly bounded in $\mathcal{E}_p$. Since $\{f_{\varepsilon_n}\}_{n\geq 1}$ converges to $f$ in $L^p(X,\mu)$ due to Proposition~\ref{P:Lip_const_estimate}, it follows from~\eqref{E:abs_cont_Gamma_p_03} that a sequence $\{f_{\varepsilon_n}\}_{n\geq 1}$ with $\varepsilon_n\to 0$ is uniformly bounded in $(KS^{1,p}(X),(\mathcal{E}_p+\|\cdot\|_{L^p(X,\mu)})^{1/p})$. 

By reflexivity of $(KS^{1,p}(X),(\mathcal{E}_p+\|\cdot\|_{L^p(X,\mu)})^{1/p})$ we may thus extract a weakly convergent subsequence still denoted $f_{\varepsilon_n}$, which will again converge to $f$ since $f_{\varepsilon_n}\xrightarrow{n\to\infty} f$ in $L^p(X,\mu)$. From Mazur lemma, a convex combination of the $f_{\varepsilon_n}$, say $g_n$ will converge to $f$ in $(KS^{1,p}(X),(\mathcal{E}_p+\|\cdot\|_{L^p(X,\mu)})^{1/p})$.

\item Finally, because $f_\varepsilon$ is defined as a linear combination, Lemma~\ref{L:lin_abs_cont} and Lemma~\ref{L:abs_cont_partition} imply that $\Gamma_p(g_n)$ is absolutely continuous with respect to $\mu$. The absolute continuity of $\Gamma_p(f)$ now follows from Lemma~\ref{L:sequence_abs_cont}.
\end{enumerate}
\end{proof}

\subsection{Equivalence of $\Gamma_1$ with Miranda's BV measures when  $p=1$}\label{SS:BV measure}

In this section we assume that $p=1$. Our goal will be to compare the measures $\Gamma_1 (f)$ and the BV measures introduced in~\cite{miranda}. These measures were defined as follows, see~\cite[Section 3]{miranda}: For $f \in L^1(X,\mu)$ and $U \in \mathcal O$, let
\[
\| Df \| (U):=\inf_{f_n} \liminf_{n \to +\infty}  \int_U {\rm Lip }f_n d\mu,
\]
where the infimum is taken over the sequences of locally Lipschitz functions $f_n$ such that $f_n \to f$ in $L^1_{loc}(X, \mu)$. It was proved in~\cite[Theorem 3.4]{miranda} that $\| Df \| $ defines a Radon measure on $X$ for every $f \in KS^{1,1}(X)$, and the next theorem shows that it is equivalent to the 1-energy from Theorem~\ref{T:KS_as_Gamma}.

\begin{theorem}\label{T:BV_case}
There exist $C_1,C_2>0$ such that for every $f \in KS^{1,1}(X)$ and $U \in \mathcal O$ with $\bar{U}$ compact
\[
C_1 \| Df \| (U) \le \Gamma_1 (f) (U) \le C_2 \| Df \| (U).
\]
In particular for every $f \in KS^{1,1}(X)$, the measures $\| Df \|$ and $\Gamma_1 (f) $ are equivalent with bounded Radon-Nikodym derivatives.
\end{theorem}

\begin{proof}
Let $f \in KS^{1,1}(X)$ and $U \in \mathcal{O}$. It follows from Proposition  \ref{P:for_local_Kumagai_Sturm}
that for any $r>0$ and any locally Lipschitz function $f_n$,
    \begin{equation*}
     \frac{1}{r}\int_U\fint_{B(x,r)}|f_n(x)-f_n(y)|d\mu(y)\,d\mu(x)\leq C\int_{U_{\Lambda r}}|{\rm Lip }f_n|d\mu.
     \end{equation*}
Thus,
     \begin{equation*}
     \frac{1}{r}\int_U\fint_{B(x,r)}|f(x)-f(y)|d\mu(y)\,d\mu(x)\leq C \inf_{f_n} \liminf_{n \to +\infty} \int_{U_{\Lambda r}}|{\rm Lip }f_n|d\mu,
     \end{equation*}
     where the infimum is taken over the sequences of locally Lipschitz functions $f_n$ such that $f_n \to f$ in $L^1_{loc}(X, \mu)$. Therefore, for any $r>0$,
     \[
     E_{1,r}(f,U) \le C \| Df \| (U_{\Lambda r}).
     \]
     This yields that for every $U,V \in \mathcal O$ with $U \Subset V$,
     \[
     \limsup_{r \to 0} E_{1,r}(f,U) \le C \| Df \| (V)
     \]
and we deduce from~\eqref{E:local_KS_comp_Gamma} that for every $U,V \in \mathcal O$ with $U \Subset V$,
     \[
     \Gamma_1 (f) (U) \le C  \| Df \| (V).
     \]
The outer regularity of $\| Df \|$ implies that for every $U \in \mathcal O$ such that $\bar U$ is compact
      \[
     \Gamma_1 (f) (U) \le C  \| Df \| (U).
     \]
To prove the converse estimate, let $f \in KS^{1,1}(X)$ and $U \in \mathcal O$. As in \eqref{E:local_Kumagai_Sturm_cond_01} one gets
     \[
  \int_U {\rm Lip }f_\varepsilon d\mu     \leq \frac{C}{\varepsilon}\int_{U_{5\varepsilon}}\fint_{B(z,2\varepsilon)}|f(z)-f(y)|^pd\mu(y)\,d\mu(z).
     \]
     Since $f_\varepsilon \to f$ in $L^1$, it follows that for $U,V \in \mathcal O$ with $U \Subset V$
     \[
  \| Df \| (U) =    \inf_{f_n} \liminf_{n \to +\infty} \int_{U}|{\rm Lip }f_n|d\mu \le C \liminf_{\varepsilon \to 0} E_{1,\varepsilon}(f,V).
     \]
By virtue of~\eqref{E:local_KS_comp_Gamma} we obtain for every $U,V \in \mathcal O$ with $U \Subset V$
     \[
   \| Df \| (U) \le C  \Gamma_1 (f) (V).
     \]
Finally, the outer regularity of $\Gamma_1 (f)$ yields 
      \[
      \| Df \| (U) \le C  \Gamma_1 (f) (U)
     \]
for every $U \in \mathcal O$ such that $\bar U$ is compact.
\end{proof}

\section{Mosco convergence}\label{Mosco section}
While the concept of Mosco convergence was originally introduced by Mosco in the context of Dirichlet forms~\cite{Mos94}, it readily extends to more general functionals as described in Section 5 of \cite{ambrosio}. 

\begin{defn}[Mosco convergence]\label{D:Mosco_convergence}
A sequence of functionals $\{E_n\colon L^p(X,\mu)\to [-\infty,\infty]\}_{n\geq 1}$ is said to Mosco-converge to $E\colon L^p(X,\mu)\to [-\infty,\infty]$ if and only if
\begin{enumerate}[label=(\roman*)]
\item For every $f \in L^p(X,\mu)$ and every sequence $f_n$ that converges to $f$ \emph{weakly} in $L^p(X,\mu)$ it holds that
\[
E(f) \le \liminf_{n \to +\infty} E_n (f_n).
\]
\item For every $f \in L^p(X,\mu)$ there exists a sequence $f_n$ converging to $f$ strongly in $L^p(X,\mu)$ such that
\[
\limsup_{n \to +\infty} E_n (f_n)\le E(f).
\]
\end{enumerate}
\end{defn}
By definition, $\Gamma$-convergence is weaker than Mosco convergence. However we will show in Section~\ref{SS:Mosco_convergence} that both convergences are equivalent when the sequence of forms is \emph{asymptotically compact}.
More precisely, we prove in Theorem~\ref{mosco convergence} that the sequence of Korevaar-Schoen energies $\{E_{p,r_n}\colon KS^{1,p}(X)\to\mathbb{R}\}_{n\geq 1}$ from~\eqref{E:def_Epr} Mosco converges to $\mathcal{E}_p\colon KS^{1,p}(X)\to\mathbb{R}$ when the underlying space is compact. The latter assumption on the space will thus hold throughout the section.

\begin{assump}
    The underlying space $(X,d,\mu)$ is compact.
\end{assump}

\subsection{Rellich-Kondrachov}
In the present context, a sequence of forms $\{E_{p,r_n}\}_{n\geq 0}$ is said to be asymptotically compact if for a sequence of positive numbers $\{\eps_n\}_{n\geq 0}$ with $\lim_{n\to\infty}\eps_n=0$, any sequence $\{f_n\}_{n\geq 1}\subset L^p(X,\mu)$ with 
\begin{equation*}
\liminf_{n\to\infty} (E_{p,\varepsilon_n}(f_n)+\|f_n\|_{L^p(X,\mu)}^p)<\infty
\end{equation*}
has a subsequence that converges strongly in $L^p(X,\mu)$. As pointed out in Remark~\ref{R:pPI_Lip_vs_ug}, the $(p,p)$-Poincar\'e inequality~\eqref{E:pPI_Lip} is equivalent to that same equality with upper gradients $\{g_n\}_{n\geq 1}$ on the right hand side. Under this assumption and since $X$ is compact, we know from~\cite[Theorem 8.1]{HK00} that there exists $k>1$ such that any sequence $\{f_n\}_{n\geq 1}$ in $KS^{1,p}(X)$ with
\begin{equation}\label{E:Koskela-Hajlasz-cond}
\sup_{n\geq 1}\big(\|f_n\|_{L^1(X,\mu)}+\|g_n\|_{L^p(X,\mu)}\big)<\infty    
\end{equation}
contains a subsequence that converges in $L^\alpha(X,\mu)$ for any $1\leq \alpha<kp$. This observation will lead as in~\cite[Lemma 3.8]{ARB23} to asymptotic compactness.

\begin{thm}\label{T:Rellich_Kondrachov}
Let $\{\eps_n\}_{n\geq 0}$ with $\lim\limits_{n\to\infty}\eps_n=0$. Any sequence $\{f_n\}_{n\geq 1}$ in $KS^{1,p}(X)$ such that
\begin{equation}\label{E:Rellich_Kondrachov_cond}
\liminf_{n\to\infty} \Big(E_{p,\varepsilon_n}(f_n)+\|f_n\|_{L^p(X,\mu)}^p\Big)<\infty
\end{equation}
contains a subsequence that converges strongly in $L^p(X,\mu)$.
\end{thm}
\begin{proof}
    First, in view of~\eqref{E:Rellich_Kondrachov_cond} we may extract a subsequence $\{f_{n_k}\}_{k\geq 1}$ such that
    \[
    \sup_{k \ge 1} \Big( E_{p,\varepsilon_{n_k}}(f_{n_k})+ \|f_{n_k}\|_{L^p(X,\mu)}^p\Big) <+\infty.
    \]
    Second, consider the Lipschitz approximating sequence $\{f_{n_k,\varepsilon_{n_k}/2}\}_{n\geq 1}$ defined as in~\eqref{E:def_Lip_approx}. By construction, $f_{n_k,\varepsilon_{n_k}/2}$ is locally Lipschitz and thus~\cite[Lemma 6.2.6]{HKST15} implies $g_{n_k,\varepsilon_{n_k}/2}={\rm Lip}f_{n_k,\varepsilon_{n_k}/2}$ is an upper gradient of $f_{n_k,\varepsilon_{n_k}/2}$. It now follows from~\eqref{E:local_Kumagai_Sturm_cond_01} and Lemma~\ref{L:Kumagai-Sturm_condition} that
\begin{align}
    \|g_{n_k,\varepsilon_{n_k}/2}\|_{L^p(X,\mu)}^p&=\|{\rm Lip}f_{n_k,\varepsilon_{n_k}/2}\|_{L^p(X,\mu)}^p=\int_X|{\rm Lip}f_{n_k,\varepsilon_{n_k}/2}|^pd\mu(y)\notag\\
    &\leq \frac{C}{\varepsilon_{n_k}^p}\int_X\fint_{B(x,\varepsilon_{n_k})}|f_{n_k}(y)-f_{n_k}(x)|^pd\mu(y)\,d\mu(x)\notag\\
    &=CE_{p,\varepsilon_{n_k}}(f_{n_k})\leq C\liminf_{n\to\infty}E_{p,\varepsilon_n}(f_n).\label{E:Rellich_Kondrachov_01}
\end{align}
    In addition, the sequence $\{f_{n_k}\}_{k\geq 1}$ is bounded in $L^1(X,\mu)$ because $X$ is compact, and so is $f_{n_k,\varepsilon_{n_k}/2}$ due to its definition, c.f.~\eqref{E:def_Lip_approx}. Together with~\eqref{E:Rellich_Kondrachov_01}, it now follows from~\eqref{E:Rellich_Kondrachov_cond} that
    \[
    \sup_{k\geq 1}\big(\|f_{n_k,\varepsilon_{n_k}/2}\|_{L^1(X,\mu)}+\|g_{n_k,\varepsilon_{n_k}/2}\|_{L^p(X,\mu)}\big)<\infty.
    \]
    By virtue of~\cite[Theorem 8.1]{HK00} it is possible to extract yet another subsequence, which we still denote for simplicity $\{f_{n_k,\varepsilon_{n_k}}\}_{k\geq 1}$, that converges in $L^p(X,\mu)$ to some function $f\in L^{p}(X,\mu)$.\\
    Finally, we show that $\{f_{n_k}\}_{k\geq 1}$ also converges to $f$. Writing
    \begin{equation}\label{E:Rellich_Kondrachov_02}
        \|f_{n_k}-f\|_{L^p(X,\mu)}\leq \|f_{n_k}-f_{n_k,\varepsilon_{n_k}}\|_{L^p(X,\mu)}+\|f_{n_k,\varepsilon_{n_k}}-f\|_{L^p(X,\mu)}
    \end{equation}
    it only remains to prove that the first term above vanishes as $k\to\infty$. In the same manner as~\eqref{E:abs_cont_Gamma_p_03},
    \begin{align*}
       \|f_{n_k}-f_{n_k,\varepsilon_{n_k}}\|_{L^p(X,\mu)}&\leq C\int_X\bigg(\fint_{B(x,6\varepsilon_{n_k})}|f_{n_k}(x)-f_{n_k}(y)|\,d\mu(y)\bigg)^pd\mu(x) \\
        & \le C \varepsilon_{n_k}^p  \sup_{k \ge 1}  E_{p,\varepsilon_{6n_k}}(f_{n_k}) \to_{k \to +\infty} 0
    \end{align*}   
       \end{proof}

\subsection{Mosco convergence}\label{SS:Mosco_convergence}
As in the original argument by Mosco for Dirichlet forms, c.f.~\cite[Lemma 2.3]{Mos94}, we observe that $\Gamma$-convergence and Mosco convergence are equivalent for asymptotically compact sequences.
\begin{lem}\label{L:Mosco_iff_Gamma}
Let $\{E_n\colon L^p(X,\mu)\to\mathbb{R} \}_{n\geq 1}$ be a sequence of functionals that is asymptotically compact. Then the sequence $\Gamma$-converges to a functional $E\colon L^p(X,\mu)\to\mathbb{R}$ if and only if it does in the Mosco sense.
\end{lem}
\begin{proof}
    By definition, Mosco convergence implies $\Gamma$-convergence, hence it only remains to prove the converse. In particular, it suffices to show that any sequence $\{f_n\}_{n\geq 1}$ that converges strongly to some $f\in L^p(X,\mu)$ satisfies 
    \begin{equation}\label{E:Mosco_cond_a}
        E(f)\leq \liminf_{n\to\infty} E_n(f_n).
    \end{equation}
    Assume to the contrary that there is a strongly convergent sequence $\{f_n\}_{n\geq 1}$ for which
    \begin{equation}\label{E:Mosco_cond_a_contradict}
        E(f)\geq \liminf_{n\to\infty} E_n(f_n).
    \end{equation}
    Possibly extracting a subsequence, we find $\{f_{n_k}\}_{k\geq 1}$ with the property that 
    \[
    \liminf_{k\to\infty}\Big(E_{n_k}(f_{n_k})+\|f_{n_k}\|_{L^p(X,\mu)}\Big)<\infty.
    \]
    Because the sequence $\{E_{n_k}\}_{k\geq 1}$ is asymptotically compact, by definition it contains yet another subsequence that converges strongly in $L^p(X,\mu))$ to some $\tilde{f}\in L^p(X,\mu)$. For simplicity we denote that subsequence again by $\{f_{n_k}\}_{n\geq 1}$ and observe that it converges weakly to $f$ by assumption, whence $f=\tilde{f}$. But then~\eqref{E:Mosco_cond_a_contradict} implies
    \[
    E(f)\geq \liminf_{k\to\infty} E_{n_k}(f_{n_k})
    \]
    which contradicts the $\Gamma$-convergence of $\{E_{n_k}\}_{k\geq 1}$.
\end{proof}
The latter lemma now implies that the $p$-energy $(\mathcal{E}_p,KS^{1,p}(X))$ is in fact a Mosco limit.
\begin{thm}\label{mosco convergence}
If the underlying space $X$ is compact, the sequence of functionals $\{E_{p,r_n}\}_{n\geq 1}$ Mosco converges to the $p$-energy $(\mathcal{E}_p,KS^{1,p}(X))$.
\end{thm}
\begin{proof}
The claim follows from Theorem~\ref{T:KS_as_Gamma}, Theorem~\ref{T:Rellich_Kondrachov} and  Lemma~\ref{L:Mosco_iff_Gamma}.
\end{proof}

\bibliographystyle{amsplain}
\bibliography{p-energy-Cheeger_refs}

\providecommand{\bysame}{\leavevmode\hbox to3em{\hrulefill}\thinspace}
\providecommand{\MR}{\relax\ifhmode\unskip\space\fi MR }
\providecommand{\MRhref}[2]{%
  \href{http://www.ams.org/mathscinet-getitem?mr=#1}{#2}
}
\providecommand{\href}[2]{#2}
\begin{thebibliography}{10}

\bibitem{ARB23}
P.~Alonso~Ruiz and F.~Baudoin, \emph{Mosco convergence of {K}orevaar-{S}choen
  type energy funtionals}, Ann. Sc. Norm. Super. Pisa (2023), to appear.

\bibitem{BV2}
P.~Alonso-Ruiz, F.~Baudoin, L.~Chen, L.~Rogers, N.~Shanmugalingam, and
  A.~Teplyaev, \emph{Besov class via heat semigroup on {D}irichlet spaces {II}:
  {BV} functions and {G}aussian heat kernel estimates}, Calc. Var. Partial
  Differential Equations \textbf{59} (2020), no.~3, Paper No.103, 32.

\bibitem{ambrosiosurvey}
L.~Ambrosio, M.~Colombo, and S.~Di~Marino, \emph{Sobolev spaces in metric
  measure spaces: reflexivity and lower semicontinuity of slope}, Variational
  methods for evolving objects, Adv. Stud. Pure Math., vol.~67, Math. Soc.
  Japan, [Tokyo], 2015, pp.~1--58. \MR{3587446}

\bibitem{ambrosio}
L.~Ambrosio and S.~Honda, \emph{New stability results for sequences of metric
  measure spaces with uniform {R}icci bounds from below}, Measure theory in
  non-smooth spaces, Partial Differ. Equ. Meas. Theory, De Gruyter Open,
  Warsaw, 2017, pp.~1--51. \MR{3701735}

\bibitem{Bau22}
F.~Baudoin, \emph{Korevaar-{S}choen-{S}obolev spaces and critical exponents in
  metric measure spaces}, arXiv (2022), no.~arXiv:2207.12191.

\bibitem{Che99}
J.~Cheeger, \emph{Differentiability of {L}ipschitz functions on metric measure
  spaces}, Geom. Funct. Anal. \textbf{9} (1999), no.~3, 428--517.

\bibitem{DMas93}
G.~Dal~Maso, \emph{An introduction to {$\Gamma$}-convergence}, Progress in
  Nonlinear Differential Equations and their Applications, vol.~8,
  Birkh\"{a}user Boston, Inc., Boston, MA, 1993.

\bibitem{dGF75}
E.~De~Giorgi and T.~Franzoni, \emph{Su un tipo di convergenza variazionale},
  Atti Accad. Naz. Lincei Rend. Cl. Sci. Fis. Mat. Nat. (8) \textbf{58} (1975),
  no.~6, 842--850.

\bibitem{GiorgiLetta}
E.~De~Giorgi and G.~Letta, \emph{Une notion g\'{e}n\'{e}rale de convergence
  faible pour des fonctions croissantes d'ensemble}, Ann. Scuola Norm. Sup.
  Pisa Cl. Sci. (4) \textbf{4} (1977), no.~1, 61--99. \MR{466479}

\bibitem{FOT11}
M.~Fukushima, Y.~Oshima, and M.~Takeda, \emph{Dirichlet forms and symmetric
  {M}arkov processes}, extended ed., De Gruyter Studies in Mathematics,
  vol.~19, Walter de Gruyter \& Co., Berlin, 2011.

\bibitem{Gor22}
W.~G\'{o}rny, \emph{Bourgain-{B}rezis-{M}ironescu approach in metric spaces
  with {E}uclidean tangents}, J. Geom. Anal. \textbf{32} (2022), no.~4, Paper
  No. 128, 22.

\bibitem{HK00}
P.~Haj{\l}asz and P.~Koskela, \emph{Sobolev met {P}oincar\'{e}}, Mem. Amer.
  Math. Soc. \textbf{145} (2000), no.~688, x+101.

\bibitem{HP21}
B.-X. {Han} and A.~{Pinamonti}, \emph{{On the asymptotic behaviour of the
  fractional Sobolev seminorms in metric measure spaces:
  Bourgain-Brezis-Mironescu's theorem revisited}}, arXiv:2110.05980 (2021).

\bibitem{HKST15}
J.~Heinonen, P.~Koskela, N.~Shanmugalingam, and J.~T. Tyson, \emph{Sobolev
  spaces on metric measure spaces}, New Mathematical Monographs, vol.~27,
  Cambridge University Press, Cambridge, 2015, An approach based on upper
  gradients.

\bibitem{HPS04}
P.~E. Herman, R.~Peirone, and R.~S. Strichartz, \emph{{$p$}-energy and
  {$p$}-harmonic functions on {S}ierpinski gasket type fractals}, Potential
  Anal. \textbf{20} (2004), no.~2, 125--148.

\bibitem{Hin10}
M.~Hino, \emph{Energy measures and indices of {D}irichlet forms, with
  applications to derivatives on some fractals}, Proc. Lond. Math. Soc. (3)
  \textbf{100} (2010), no.~1, 269--302.

\bibitem{Kig23}
J.~Kigami, \emph{Conductive homogeneity of compact metric spaces and
  construction of {$p$}-energy}, Memoirs of the European Mathematical Society,
  vol.~5, European Mathematical Society (EMS), Berlin, [2023] \copyright 2023.

\bibitem{KS93}
N.~J. Korevaar and R.~M. Schoen, \emph{Sobolev spaces and harmonic maps for
  metric space targets}, Comm. Anal. Geom. \textbf{1} (1993), no.~3-4,
  561--659.

\bibitem{MMS16}
N.~Marola, M.~Miranda, Jr., and N.~Shanmugalingam, \emph{Characterizations of
  sets of finite perimeter using heat kernels in metric spaces}, Potential
  Anal. \textbf{45} (2016), no.~4, 609--633.

\bibitem{miranda}
M.~Miranda, Jr., \emph{Functions of bounded variation on ``good'' metric
  spaces}, J. Math. Pures Appl. (9) \textbf{82} (2003), no.~8, 975--1004.
  \MR{2005202}

\bibitem{Mos94}
U.~Mosco, \emph{Composite media and asymptotic {D}irichlet forms}, J. Funct.
  Anal. \textbf{123} (1994), no.~2, 368--421.

\bibitem{MS23}
M.~Murugan and R.~Shimizu, \emph{First-order {S}obolev spaces, self-similar
  energies and energy measures on the {S}ierpi\'nski carpet}, arXiv (2023),
  no.~arXiv:2308.06232.

\bibitem{Rud87}
W.~Rudin, \emph{Real and complex analysis}, third ed., McGraw-Hill Book Co.,
  New York, 1987.

\bibitem{Sha00}
N.~Shanmugalingam, \emph{Newtonian spaces: an extension of {S}obolev spaces to
  metric measure spaces}, Rev. Mat. Iberoamericana \textbf{16} (2000), no.~2,
  243--279.

\bibitem{SW04}
R.~S. Strichartz and C.~Wong, \emph{The {$p$}-{L}aplacian on the {S}ierpinski
  gasket}, Nonlinearity \textbf{17} (2004), no.~2, 595--616.

\end{thebibliography}

\noindent
\textbf{Fabrice Baudoin}: \url{fbaudoin@math.au.dk}\\
Department of Mathematics,
Aarhus University,
Denmark

\medskip

\noindent
\textbf{Patricia Alonso Ruiz}: \url{paruiz@tamu.edu}\\
Department of Mathematics
Texas A\&M University
USA
\end{document}